\documentclass[12pt,reqno]{amsart}
\usepackage{fullpage}
\usepackage{amssymb,amsfonts}
\usepackage[all,arc]{xy}
\usepackage{bbm}
\usepackage{enumerate}
\usepackage{mathrsfs}
\usepackage{color}   %May be necessary if you want to color links
\usepackage[pagebackref=true, colorlinks]{hyperref}
\usepackage{marginnote}
\hypersetup{
    colorlinks=true, %set true if you want colored links
    linktoc=all,     %set to all if you want both sections and subsections linked
    linkcolor=blue,  %choose some color if you want links to stand out
    citecolor=blue
}
\usepackage{amsthm}
\usepackage{amsmath}
\usepackage{graphicx}
\usepackage{wasysym}
\usepackage{pgf,tikz}
\usepackage{mathtools}
\usepackage{tikz-cd}

\usetikzlibrary{positioning}

\newcommand{\bbC}{\mathbb C}

\newcommand{\N}{\mathbb N}
\newcommand{\bbZ}{\mathbb Z}

\newcommand{\R}{\mathbb R}

\newcommand\floor[1]{\lfloor#1\rfloor}

\DeclareMathOperator{\SU}{SU}
\DeclareMathOperator{\SO}{SO}
\DeclareMathOperator{\SL}{SL}
\DeclareMathOperator{\Sp}{Sp}

\newcommand{\fg}{\mathfrak g}
\newcommand{\fa}{\mathfrak a}

\newcommand{\fn}{\mathfrak n}
\newcommand{\fk}{\mathfrak k}

\newcommand{\cX}{\mathcal{X}}

\newcommand{\G}{\Gamma}
\newcommand{\bk}{\backslash}
\newcommand{\vf}{\varphi}
\newcommand{\<}{\left\langle}
\renewcommand{\>}{\right\rangle}
\newcommand{\cA}{\mathcal{A}}

\newcommand{\cC}{\mathcal{C}}
\newcommand{\cH}{\mathcal{H}}
\newcommand{\cM}{\mathcal{M}}

\newtheorem{Thm}{Theorem}[section]
\newtheorem{Prop}[Thm]{Proposition}
\newtheorem{Lem}[Thm]{Lemma}
\newtheorem{Cor}[Thm]{Corollary}

\theoremstyle{definition}
\newtheorem{Def}[Thm]{Definition}
\newtheorem{Rem}{Remark}

\theoremstyle{remark}

\theoremstyle{definition}

\numberwithin{equation}{section}

\title{Shrinking targets problems for flows on homogeneous spaces}
\author{Dubi Kelmer}
\address{Department of Mathematics, Boston College, Chestnut Hill MA 02467-3806, USA}
\email{kelmer@bc.edu}
\author{Shucheng Yu}
\address{Department of Mathematics, Boston College, Chestnut Hill MA 02467-3806, USA}
\email{shucheng.yu@bc.edu}
\thanks{This work was partially supported by NSF grant DMS-1401747 and NSF CAREER grant DMS-1651563.}

\begin{document}

\begin{abstract}
We study shrinking targets problems for discrete time flows on a homogenous space $\G\bk G$ with $G$ a semisimple group and $\G$ an irreducible lattice. Our results apply to both diagonalizable and unipotent flows, and apply to very general families of shrinking targets. As a special case, we establish logarithm laws for cusp excursions of unipotent flows answering a question of Athreya and Margulis.\end{abstract}
 \maketitle

 \section{Introduction}
Consider an ergodic dynamical system given by the iteration of a measure preserving map $T: \cX\to\cX$ on a probability space $(\cX,\mu)$. From ergodicity, it follows that generic orbits become dense, and shrinking target problems are a way to quantify the rate. That is, to determine how fast we can make a sequence of targets shrink so that a typical orbit will keep hitting the targets infinitely often.

A natural bound for this rate comes from the easy half of Borel-Cantelli, stating that for any sequence of sets, $\{B_m\}_{m\in\N}$, if
$\sum_{m=1}^{\infty} \mu(B_m)<\infty$ then for a.e. $x\in \cX$ from some point on $T^m x\not\in B_m$.
For chaotic dynamical systems, it could be expected that this bound is sharp, and much work has gone into proving this in various examples of fast mixing dynamical systems (under some regularity restrictions on the shrinking sets).
In particular, this was done for shrinking cusp neighborhoods of homogenous spaces \cite{Sullivan1982, KleinbockMargulis1999,GorodnikShah11,AthreyaMargulis09,KelmerMohammadi12, AthreyaMargulis14,Yu17},
and more generally for shrinking metric balls in a metric space \cite{ChernovKleinbock01,Dolgopyat04,Galatolo07,KleinbockZhao2017}.

Recently, in \cite{Kelmer17b}, the first author introduced a new method for attacking this problem for discrete time homogenous flows on (the frame bundle of) finite volume hyperbolic manifolds. This method works for any monotone family of shrinking targets in the hyperbolic manifold, and applies also for unipotent flows with arbitrarily slow polynomial mixing rate. In this paper we adapt this method to treat the general case of discrete time homogenous flows on a homogenous space $\cX=\G\bk G$ with $G$ a connected semisimple Lie group with finite center and no compact factors, and $\G$ an irreducible lattice.

\subsection{General setup, terminology, and notations}
Let $G$ denote a connected semisimple Lie group with finite center and no compact factors, let $\G\leq G$ be an irreducible lattice, and let $\mu$ denote the $G$-invariant probability measure on $\cX=\G\bk G$, coming from the Haar measure of $G$.
We fix once and for all a maximal compact subgroup $K\leq G$ and denote by $\cH=G/K$ the corresponding symmetric space. We say that a subset $B\subseteq \cX$ is spherical if it is invariant under the right action of $K$ and we identify spherical sets as subsets of the locally symmetric space $\G\bk\cH$.

One-parameter flows on $\cX=\G\bk G$ are given by the right action of one-parameter subgroups of $G$.
Explicitly, the one-parameter group generated by an element, $X_0$, in the Lie algebra, $\fg=\mathrm{Lie}(G)$, is given by
$\{h_t=\exp(tX_0):t\in \R\}$.   The corresponding discrete time flow is then given by the action of the discrete subgroup  $H=\{h_m\}_{m\in \bbZ}$.
%For any $m\in \N$ denote by
%\begin{equation}\label{e:H_m}
%H_m=\{h_j: 0\leq j\leq m\}
%\end{equation}
We will always assume that the subgroup is unbounded and recall that, by Moore's ergodicity theorem,  the action of any unbounded subgroup is ergodic and mixing.
%Given such a discrete time flow, for any $x\in \cX$ and $k\in \N$ we denote by $xH_k=\{xh_n: 0\leq n\leq k\}$ the corresponding truncated forward orbit.

We say that a family $\{B_t\}_{t>0}$ is a monotone family of shrinking targets if $B_t\subseteq B_s$ when $t\geq s$ and $\mu(B_t)\to 0$, and we say it is a family of spherical shrinking targets if all sets are spherical.
Given an unbounded discrete time one-parameter flow $\{h_m\}_{m\in \mathbb{Z}}$, following \cite{ChernovKleinbock01}, we say that a sequence of sets $\{B_m\}_{m\in \N}$
%with $\sum_m \mu(B_m)=\infty$
is \textit{Borel-Cantelli} (BC) for $\{h_m\}_{m\in\mathbb{Z}}$, if for a.e. $x\in\cX$ the set
$\{m\in \N\ |\ xh_m\in B_m\}$ is unbounded, and we say it is \textit{strongly Borel-Cantelli} (sBC) if for a.e. $x\in \cX$
\begin{equation}\label{e:sBC}
\lim_{m\to\infty}\frac{\#\{1\leq j\leq m: xh_j\in B_j\}}{\sum_{1\leq j\leq m}\mu(B_j)}=1.\end{equation}
%Following \cite{Athreya2009},
We say that a collection of sets is (strongly)  Borell-Cantelli (resp. monotone sBC) if any sequence (resp. monotone sequence) from this collection with divergent measure is (strongly) Borell-Cantelli.

In what follows we adopt the notation $A(t)\ll B(t)$ or $A(t)=O(B(t))$ to indicate that there is a constant $c>0$ such that $A(t)\leq cB(t)$, and we write $A(t)\asymp B(t)$ to indicate that
$A(t)\ll B(t)\ll A(t)$. The implied constants may always depend on the group, $G$, the lattice, $\G$, and the flow that we think of as fixed. We will use subscripts to indicate the dependance of the implied constants on any additional parameters.

\subsection{Logarithm laws}
Our first result establishes a logarithm law for the first hitting time function (see \cite{GalatoloKim07} for the relation between such logarithm laws and Borel-Cantelli properties).
Given an unbounded discrete time flow and a subset $B\subseteq \cX$ the first hitting time function is defined for $x\in \cX$ by
\begin{equation}\label{e:hitting}
\tau_{B}(x)=\min\{m\in \N: xh_m\in B\}.
\end{equation}
\begin{Thm}\label{t:hittingtime}
Assume that either $G$ has property $(T)$, or that $G$ is of real rank one.
Let $\{B_t\}_{t>0}$ denote a monotone family of spherical shrinking targets in $\cX=\G\bk G$. Let $\{h_m\}_{m\in \bbZ}$ denote an unbounded discrete time flow on $\cX$. Then for a.e. $x\in \cX$
\begin{equation}
\label{hit}\lim_{t\to\infty}\frac{\log(\tau_{B_t}(x))}{-\log(\mu(B_t))}=1.
\end{equation}
\end{Thm}
\begin{Rem}\label{r:conditional}
Since all simple groups without property $(T)$ are of real rank one, Theorem \ref{t:hittingtime} holds for all noncompact connected simple groups with finite center. In fact, the only groups of higher rank without property $(T)$ are (almost) products $G=\prod G_i$ with at least one of the factors a simple rank one group without property $(T)$. In this case, the same result still holds, unless the flow is a unipotent flow on one of those rank one factors.  Even then, according to the congruence subgroup conjecture (see e.g. \cite{Rapinchuk92,Raghunathan04}), for higher rank groups the only irreducible lattices are congruence lattices.  For congruence lattices, if we further assume the Selberg-Ramanujan Conjecture  (see \cite{Shahidi04,Sarnak05}), the resulting bounds on decay of matrix coefficients are sufficient to obtain the same result also in these cases (in fact, with the exception of $G_i$ locally isomorphic to $\SL_2(\R)$ for some $i$ the unconditional bounds towards the Selberg-Ramanujan Conjecture obtained in  \cite{KimSarnak03, BlomerBrumley11,BergeronClozel13,BergeronClozel17} are already sufficient). Hence, conditional on these conjectures our result holds in full generality.
\end{Rem}

%\begin{Rem}
%Instead of the first hitting time we can also consider the $i$th hitting time function
%$$\tau_B^i(x)=\inf\{m\in \N: |xH_m^+\cap B|=i\}.$$
%We note that the same result, with essentially the same proof, holds if we replace $\tau_{B_t}$ with $\tau^i_{B_t}$ for any fixed $i\in \N$.
%\end{Rem}

Since spherical sets in the homogenous space, $\G\bk G$, can be naturally identified with subsets of the locally symmetric space, $\G\bk \cH$, our result holds for any monotone sequence of shrinking targets in $\G\bk \cH$. A special case of such a sequence, that has received much attention, occurs when $\G\bk \cH$ is not compact and the shrinking sets are cusp neighborhoods  \cite{Sullivan1982, KleinbockMargulis1999,AthreyaMargulis09, KelmerMohammadi12,AthreyaMargulis14,Yu17}.   For these problems, logarithm laws can be expressed in terms of a distance function measuring how far out a point is in the cusp. Explicitly, given a continuous distance function
on $\G\bk \cH$ that we lift to a $K\times K$-invariant function $d(\cdot,\cdot)$ on $\G\bk G\times \G\bk G$,  we can define spherical shrinking cusp neighborhoods of $\cX$ by
$$B_{t}=\{x\in \cX:  d(x,x_0)>t\},$$
where $x_0\in \cX$ is some fixed base point.
We assume that the measure of these cusp neighborhoods decay exponentially with rate $\varkappa>0$ in the sense that
\begin{equation}\label{e:cuspdecay}
\lim_{t\to\infty} \frac{-\log\mu(B_t)}{t}=\varkappa.
\end{equation}
In particular, this holds when the distance function\footnote{While these distance function are defined on $\G\bk G$ and are not necessarily lifts from a distance function on $\G\bk \cH$, the corresponding cusp neighborhoods can always be approximated by spherical cusp neighborhoods as explained in Remark \ref{r:sphericalcusps} below.} is determined by a right $G$-invariant, bi $K$-invariant Riemannian
metric on $G$ as in \cite{KleinbockMargulis1999}, or more generally by a norm-like pseudometric on $G$ as in  \cite{AthreyaMargulis14}.
For any such distance function, a consequence of Theorem \ref{t:hittingtime} applied to this family of shrinking targets is the following.
\begin{Cor}\label{loglaws}
For $G$ a connected semisimple Lie group with finite center and no compact factors, and $\G\leq G$ an irreducible lattice,
for any unbounded one-parameter flow $\{h_t\}_{t\in\mathbb R}$ on $\cX= \G\bk G$, for a.e. $x\in \cX$
\begin{equation}\label{e:loglawcusp}
\limsup_{t\to\infty}\frac{d(xh_t,x_0)}{\log t}=\frac{1}{\varkappa}.
\end{equation}
\end{Cor}
For diagonalizable flows, the logarithm law for cusp excursions \eqref{e:loglawcusp} was established by Sullivan \cite{Sullivan1982} for $G=\SO_0(d,1)$, and by Kleinbock and Margulis \cite{KleinbockMargulis1999} in general.
For unipotent flows, Athreya and Margulis \cite{AthreyaMargulis14} showed that there is some $c\in (0,1]$ so that for a.e. $x\in \cX$,
$\limsup_{t\to\infty}\frac{d(xh_t,x_0)}{\log t}=\frac{c}{\varkappa}$, and raised the question if it is always the case that the constant $c=1$. This question was previously answered affirmatively in some specific examples, such as unipotent flows on the space of lattices and on hyperbolic manifolds \cite{AthreyaMargulis09,KelmerMohammadi12,Yu17}. Our result settles this problem and gives an affirmative answer to their question in general. %when $G$ is a simple group, and in complete generality assuming the Ramanujan-Selberg conjecture.

\subsection{Summable decay of matrix coefficients}
A key ingredient in the proof of  Theorem \ref{t:hittingtime} is a uniform rate on decay of spherical matrix elements under one-parameter flows.
We say $\vf\in L^2(\G\bk G)$ is spherical if it is invariant under the action of $K$, and we identify the space of spherical functions with $L^2(\G\bk \cH)=L^2(\G\bk G/K )$.
Let $\pi$ denote the regular representation of $G$ on $L^2(\G\bk G)$ and let $L^2_0(\G\bk G)$ denote the space orthogonal to the constant functions. For any $\psi,\vf\in L^2(\G\bk G)$ and $g\in G$ the corresponding matrix element is given by
$$\<\pi(g)\vf,\psi\>=\int_{\G\bk G}\vf(xg)\overline{\psi(x)}d\mu(x).$$
By using the results of Oh \cite{Oh98} on effective property $(T)$, we show the following uniform bound on decay of matrix coefficients for one-parameter flows.
\begin{Thm}\label{t:edecay}
Assume that  $G$ has property $(T)$ and let $\{h_t\}_{t\in \R}$ denote an unbounded one-parameter subgroup.  For all sufficiently small $\epsilon>0$,
for any spherical $\vf,\psi\in L^2_0(\G\bk G)$ and for all $|t|\geq 1$ we have
$$|\<\pi(h_t)\vf,\psi\>|\ll_\epsilon\frac{\|\vf\|_2\|\psi\|_2}{|t|^{1-\epsilon}}.$$
\end{Thm}
%\begin{Rem}
%For $G$ simple, the only groups without  property $(T)$ are of real rank one.
% in which case we no longer have such a uniform rate for the decay of matrix coefficients. Nevertheless, in these cases we can use the analysis in \cite{} to obtain a similar result using the spectral decomposition.
%\end{Rem}
%\begin{Rem}
%If we remove the assumption that the functions are spherical we get a similar result with the $L^2$-norm on the right hand side replaced by appropriate Sobolev norms. However, for our application it is crucial that the bounds depend only on  $L^2$-norms, which is the reason we restrict our results to spherical sets.
%\end{Rem}

It is remarkable that this uniform rate, which is exactly the rate needed for Theorem \ref{t:hittingtime}, can be obtained for all one-parameter flows for all semisimple groups with property $(T)$. The rate of decay here is $\epsilon$ away from being summable, and, while this  is good enough to establish a logarithm law, for other applications it just falls short. In many cases, however, we can establish a slightly better rate. In order to distinguish these cases we use the following terminology.
\begin{Def}
We say that a one-parameter flow on $\G\bk G$ has \em{Summable Decay} (or  SD) if there is $\eta>1$ %and a constant $C>0$
such that for any spherical $\vf,\psi\in L^2_0(\G\bk G)$, for all $|t|\geq 1$ we have
$$|\<\pi(h_t)\vf,\psi\>|\ll_\eta\frac{\|\vf\|_2\|\psi\|_2}{|t|^{\eta}}.$$
\end{Def}
For SD flows we can show the stronger result implying that the bound coming from Borel-Cantelli is sharp.
\begin{Thm}\label{t:SBC}
For any discrete time one-parameter SD flow on $\Gamma\backslash G$, the collection of spherical subsets is strongly Borel-Cantelli.
%Then for any sequence of spherical measurable sets $\{B_n\}_{n\in \N}$  in $\G\bk G/K$ if $\sum_n \mu(B_n)=\infty$ then for a.e. $x\in \cX$
%$$\lim_{m\to\infty} \frac{\#\{n\leq m: xh_n\in B_n\}}{\sum_{n\leq m} \mu(B_n)}=1,$$
%so in particular, the set $\{n:xh_n\in B_n\}$ is unbounded for a.e. $x$.
\end{Thm}
%\begin{Rem}
%Note that we do not assume here that the sequence $\{B_n\}_{n\in \N}$ is monotone, or even that $\mu(B_n)$ decays monotonously. Using the terminology of \cite{Athreya}, this shows that for an SD flow, any sequence of spherical sets with $\sum_n \mu(B_n)=\infty$ satisfies is strongly Borel-Cantelli, and hence, the collection of spherical sets satisfy the strong Borel-Cantelli property.
%\end{Rem}

Following this result it would be useful to categorize precisely which one-parameter homogenous flows are SD.
For diagonalizable flows, the exponential decay of matrix coefficients clearly imply summable decay.
For unipotent flows, the decay is polynomial and we don't always have summable decay.
When the group $G$ is a simple Lie group of rank $\geq 2$, the following result gives explicit conditions for when a one-parameter flow has summable decay, in terms of the restricted root system of $G$ and the adjoint representation of the generator of the flow (see section \ref{s:pLie} for more details).
\begin{Thm}\label{t:csd}
Assume that $G$ is simple with real rank $\geq 2$. If the restricted root system is of type $B_n (n\geq 4), D_n (n\geq 4), E_6, E_7, E_8$ or $F_4$, then all unbounded one-parameter flows on $\G\bk G$ are SD. When the restricted root system is not of the above types, any unipotent one-parameter flow, $h_t=\exp(tX_0)$, with $ad(X_0)^3\neq 0$ is SD.
%For $G=\SO(d+1,1)$ or $SU(d,1)$ with $n\geq 2$ if smallest Laplace eigenvalue on $L^2_0(\G\bk G/K)$ satisfies $\lambda_1>\frac{2n-1}{4}$,  then all one-parameter flows are SD.
\end{Thm}

\begin{Rem}
For example, on $\G\bk SL_3(\R)$ the unipotent flow given by  $h_t=\left(\begin{smallmatrix} 1 & 2t & 2t^2\\ 0 & 1& 2t\\ 0 &0 &1\end{smallmatrix}\right)$ is SD, while the flow given by $h_t=\left(\begin{smallmatrix} 1 & 0 &t\\ 0 & 1 & 0\\ 0 &0 &1\end{smallmatrix}\right)$ is not SD.
\end{Rem}

\begin{Rem}
For some applications it is useful to know the precise rate of decay for matrix coefficients along
unipotent flows, and we remark that
our method gives the following very explicit rate: When $h_t = \exp(tX_0)$ is unipotent, let $l$ be the largest positive integer such that $\textrm{ad}(X_0)^l\neq 0$. Then our method shows that for any spherical  $\vf,\psi\in L^2_0(\G\bk G)$,
the matrix coefficients $\<\pi(h_t)\vf,\psi\>$ are bounded by
$O_{\epsilon}\left(\frac{\|\vf\|_2\|\psi\|_2}{|t|^{l(1-\epsilon)}}\right)$ when the restricted root system is $B_n (n\geq 4), D_n (n\geq 4), E_6, E_7, E_8$ or $F_4$ and by $O_{\epsilon}\left(\frac{\|\vf\|_2\|\psi\|_2}{|t|^{l/2(1-\epsilon)}}\right)$ otherwise.
\end{Rem}

\begin{Rem}
We can also characterize SD flows when $G$ is semisimple with property (T). In such a case the flow will be SD unless the flow is essentially trivial except in one of the factors, with the restriction of the flow to this factor not SD.
\end{Rem}

%For SD flows we can show the stronger result implying that the bound coming from Borel-Cantelli is sharp.
%\begin{Thm}\label{t:SBC}
%Let $\cX=\G\bk G$ and let $\{h_t\}_{t\in \R}$ denote a one-parameter SD flow.
%Then the collection of spherical subsets of $\cX$ is strongly Borel-Cantelli.
%%Then for any sequence of spherical measurable sets $\{B_n\}_{n\in \N}$  in $\G\bk G/K$ if $\sum_n \mu(B_n)=\infty$ then for a.e. $x\in \cX$
%%$$\lim_{m\to\infty} \frac{\#\{n\leq m: xh_n\in B_n\}}{\sum_{n\leq m} \mu(B_n)}=1,$$
%%so in particular, the set $\{n:xh_n\in B_n\}$ is unbounded for a.e. $x$.
%\end{Thm}
%\begin{Rem}
%Note that we do not assume here that the sequence $\{B_n\}_{n\in \N}$ is monotone, or even that $\mu(B_n)$ decays monotonously. Using the terminology of \cite{Athreya}, this shows that for an SD flow, any sequence of spherical sets with $\sum_n \mu(B_n)=\infty$ satisfies is strongly Borel-Cantelli, and hence, the collection of spherical sets satisfy the strong Borel-Cantelli property.
%\end{Rem}

\subsection{Groups of real rank one}
For a group of real rank one with property $(T)$, every discrete time homogenous flow is SD, and hence the collection of spherical sets is sBC. For rank one groups without property $(T)$, the question if all one-parameter flows on $\G\bk G$ are SD, depends on the spectral gap of $\G$, that is, the size of the smallest non-trivial eigenvalue of the Laplacian on $L^2(\G\bk \cH)$. When the spectral gap is sufficiently large all flows are SD (see Corollary \ref{uod} below). However, there are also examples with a small spectral gap, for which unipotent flows may not be SD. Nevertheless, using a spectral decomposition it is still possible to prove that the collection of spherical sets is monotone Borell-Cantelli. This was done in \cite{Kelmer17b} for $G=\SO_0(d+1,1)$ with $d\geq 2$ and the proof is similar for $G$ locally isomorphic to $\SU(d,1)$. Explicitly, for these groups we have the following.
%\begin{Thm}
%\label{t:MBC}Let  $G=\SO(d+1,1)$ or $G=SU(d,1)$ with $n\geq 2$, and let $\{B_m\}_{m\in \N}$ denote a monotone family of spherical shrinking targets with $\sum_m \mu(B_m)=\infty$. If the sequence $\{m\mu(B_m)\}_{m\in \N}$ is bounded then for a.e. $x\in \cX$
%$$\lim_{m\to\infty} \frac{\#\{1\leq i\leq m: xh_i\in B_i\}}{\sum_{1\leq i\leq m} \mu(B_i)}=1,$$
%and otherwise there is a subsequence $m_j$ such that for a.e. $x\in \cX$
%$$\lim_{j\to\infty}\frac{\#\{1\leq i\leq m_j: xh_i\in B_{m_j}\}}{m_j \mu(B_{m_j})}=1.$$
%In particular, in both cases, for a.e. $x\in \cX$ the set $\{m\in\N: xh_m\in B_m\}$ is unbounded, so the sequence of targets is BC for this flow.
%\end{Thm}
%\begin{Rem}
%We only need the monotonicity condition for the second case. Explicitly, it follows that any sequence of spherical sets, $\{B_m\}_{m\in \N}$, with $\left\{m\mu(B_m)\right\}_{m\in\mathbb{N}}$ bounded and $\sum_m\mu(B_m)=\infty$, is strongly Borel-Cantelli.
%\end{Rem}
\begin{Thm}\label{t:MBC}
Let $G$ be locally isomorphic to $\SO(d+1,1)$ or $\SU(d,1)$ with $d\geq 2$. Then for any unbounded discrete time one-parameter flow on $\Gamma\backslash G$, the collection of spherical subsets is monotone Borel-Cantelli.
% and let $\{B_m\}_{m\in \N}$ denote a monotone family of spherical shrinking targets with $\sum_m \mu(B_m)=\infty$. If the sequence $\{m\mu(B_m)\}_{m\in \N}$ is bounded then for a.e. $x\in \cX$
%$$\lim_{m\to\infty} \frac{\#\{1\leq i\leq m: xh_i\in B_i\}}{\sum_{1\leq i\leq m} \mu(B_i)}=1,$$
%and otherwise there is a subsequence $m_j$ such that for a.e. $x\in \cX$
%$$\lim_{j\to\infty}\frac{\#\{1\leq i\leq m_j: xh_i\in B_{m_j}\}}{m_j \mu(B_{m_j})}=1.$$
%In particular, in both cases, for a.e. $x\in \cX$ the set $\{m\in\N: xh_m\in B_m\}$ is unbounded, so the sequence of targets is BC for this flow.
\end{Thm}
\begin{Rem}
In fact, our proof gives something stronger. For a sequence of spherical sets, $\{B_m\}_{m\in \N}$ (not necessarily monotone) with $\sum_m\mu(B_m)=\infty$, if we assume that $\left\{m\mu(B_m)\right\}_{m\in\mathbb{N}}$ is bounded then we can show this sequence is strongly Borel-Cantelli.
If the sequence is monotone and $\left\{m\mu(B_m)\right\}_{m\in\mathbb{N}}$ is unbounded, we can show instead that there is a subsequence $\{m_j\}$ such that for a.e. $x\in \cX$
$$\lim_{j\to\infty}\frac{\#\{1\leq i\leq m_j: xh_i\in B_{m_j}\}}{m_j \mu(B_{m_j})}=1.$$
\end{Rem}
\begin{Rem}
The case when $G$ is locally isomorphic to $\SO(2,1)%\cong \SU(1,1)\cong\SL_2(\R)
$ was considered in \cite{Kelmer17b}. In this case one needs some additional assumptions on the shrinking rate of family $\{B_m\}_{m\in \N}$ to show that it is BC for a unipotent flow.
\end{Rem}

\subsection{Orbits eventually always hitting}
Assume now that our flow is either SD or that $G$ is of real rank one, so that any monotone sequence of spherical shrinking targets $\{B_m\}_{m\in \N}$ is BC. In such cases we wish to study the subtler point, of whether the finite orbits
\begin{equation}\label{e:Hm}
xH_m^+=\{xh_j: 1\leq j\leq m\},
\end{equation}
eventually always hit or miss the targets $B_m$.
Using the terminology introduced in \cite{Kelmer17b} we say that an orbit of a point $x\in \cX$ is {\em{eventually always hitting}}  if $xH_m^+\cap B_m\neq\emptyset$ for all sufficiently large $m$, and {\em{eventually always missing}} if
$xH_m^+\cap B_m=\emptyset$ for all sufficiently large $m$. We
denote by $\cA_{\rm ah}$ and $\cA_{\rm am}$ the set of points with such orbits respectively.
For the eventually always missing set, we have the dichotomy given by the dynamical Borel-Cantelli Lemma: if $\sum_{m=1}^{\infty}\mu(B_m)<\infty$ then $\cA_{\rm am}$ is of full measure and otherwise $\cA_{\rm am}$ is a null set.

The eventually always hitting set, $\cA_{\rm ah}$, is also either a null set or a set of full measure, but in this case we are not able to establish such an explicit dichotomy.
Here we have the following partial result (extending the result of  \cite{Kelmer17b} dealing with the case of $G=\SO_0(d+1,1)$).
% the first author considered this problem for the case of $G=\SO(d+1,1)$ and obtained the following partial result: In one direction,
%if there is $c<1$ such that the set
%$\{m: m\mu(B_m)\leq c\}$ is unbounded then $\cA_{\rm ah}$ is a null set, and on the other hand
%if
%\begin{equation}\label{e:summable}
%\sum_{j=0}^\infty \frac{1}{2^{j}\mu(B_{2^j})}<\infty,
%\end{equation}
%then $\cA_{\rm ah}$ has full measure.
%The first part is very general and holds for any discrete time flow. We show here that the second part also holds in this more general setting. Explicitly, we show
\begin{Thm}\label{t:ae}
Fix an unbounded discrete time one-parameter flow on $\G\bk G$ and assume that either the flow is SD, or that $G$ is of real rank one, not locally isomorphic to $\SO(2,1)$.
Let $\{B_m\}_{m\in \N}$ denote a monotone sequence of spherical shrinking targets.
If \begin{equation}\label{e:summable}
\sum_{j=0}^\infty \frac{1}{2^{j}\mu(B_{2^j})}<\infty,
\end{equation}
Then $\cA_{\rm ah}$ is of full measure. Moreover, if we further assume that
$\mu(B_{2m})\asymp \mu(B_m)$, then for a.e. $x\in \cX$, %there is a constant $C>1$ such that
for all sufficiently large $m$
$$\#\{1\leq j\leq m: xh_j\in B_m\}\asymp m\mu(B_m).$$
%$$C^{-1}m\mu(B_m)\leq \#\{xH_m\cap B_m\}\leq C  m\mu(B_m)$$
\end{Thm}

\begin{Rem}
In the other direction, it was shown in  \cite{Kelmer17b} that for any ergodic one-parameter flow, for any monotone sequence, $\{B_m\}_{m\in \N}$, of shrinking targets, if there is $c<1$ such that the set
$\{m: m\mu(B_m)\leq c\}$ is unbounded then $\cA_{\rm ah}$ is a null set.
In particular, if we assume that $\mu(B_m)$ decays polynomially in the sense that $\mu(B_m)\asymp m^{-\eta}$ for some fixed $\eta$,
%$$C_1m^{-\eta}\leq \mu(B_m)\leq C_2m^{-\eta},$$
then Theorem \ref{t:ae}, implies that $\cA_{\rm ah}$ is a set of full measure when $\eta<1$, and a null set when $\eta>1$.
In this case, however, the same result already follows from Theorem \ref{t:hittingtime} and hence holds also for flows that are not SD.
\end{Rem}
\begin{Rem}
Finally we remark that in a recent work of Kleinbock and Wadleigh \cite{KleinbockWadleigh2017}, they give examples of dynamical systems where they manage to get an explicit dichotomy for $\cA_{\rm ah}$, and the dichotomy depends exactly on the convergence and divergence of the series of the type $(\ref{e:summable})$.
\end{Rem}

% \subsection{Dynamical Borell Cantelli}
% In analogy to the Borel-Cantelli lemma, we say that a sequence $\{B_n\}_{n\in \N}$ is Borell-Cantelli (BC) if
% $\sum_n\mu(B_n)=\infty$ and for a.e. $x\in \cX$ the set $\{n: xh_n\in B_n\}$ is unbounded, and that it is Strong Borel Cantelli (SBC) if for a.e. $x\in \cX$
% $$\lim_{n\to\infty} \frac{\#\{m\leq n: xh_m\in B_m\}}{\sum_{m\leq n} \mu(B_m)}=1.$$
%
% \begin{Thm}
% For any one-parameter flow on  $\cX=\G\bk G$ that is SD,
% any sequence of spherical sets with $\sum_m \mu(B_m)=\infty$ is SBC
% \end{Thm}
%
% \begin{Thm}
%Let $G\not\cong \SL_2(\R)$ be of real rank one , $\G\leq G$ a lattice, and fix a one-parameter flow on $\G\bk G$. Any sequence of spherical sets with $\sum_m \mu(B_m)=\infty$ is BC, and if we further assume $m\mu(B_m)$ is bounded then it is SBC.  \end{Thm}
%
\subsection*{Acknowledgements}
The authors would like to thank Dmitry Kleinbock, Farrell Brumley and Jayadev S. Athreya for helpful conversations.

\section{Shrinking target problems}
We start by taking a closer look on shrinking target problems for a general ergodic $\mathbb{Z}$-action given by the action of a group, $H=\{h_m\}_{m\in\bbZ}$, on a probability space $(\mathcal{X},\mu)$. Though we will later apply these results for the case of $\cX=\G\bk G$, in this section we will not assume anything about the space $\cX$ and the flow other than ergodicity.

\subsection{The hitting time problem}
Fix a positive integer $i$ throughout this subsection. Instead of the  first hitting time function given in \eqref{e:hitting}, we consider the more general $i$th hitting time function:
For any $x\in \mathcal{X}$ and any set $B\subset \mathcal{X}$ let
\begin{equation}\label{e:ihitting}
\tau^{i}_{B}(x):=\min\{m\in\mathbb{N}\ |\ |xH_m^+\cap B|=i\}\end{equation}
measure the time needed for the orbit of $x$ to enter the target $B$ for the $i$th time (here $xH_m^+$ is as in \eqref{e:Hm}). In preparation for Theorem \ref{t:hittingtime}, our goal in this section is to give sufficient conditions on a monotone family of shrinking targets, $\{B_t\}_{t> 0}$, implying that for a.e. $x\in \cX$
\begin{equation}\label{e:limtaui}
\lim\limits_{t\to \infty}\frac{\log \tau^{i}_{B_t}(x)}{-\log\mu(B_t)}=1.
\end{equation}

%\begin{equation}\label{upper.bound}
%\limsup\limits_{t\to\infty}\frac{\log \tau^{i}_{B_t}(x)}{-\log\mu(B_t)}\leq 1
%\end{equation}
%and
%\begin{equation}\label{lower.bound}
%\liminf\limits_{t\to \infty}\frac{\log \tau^{i}_{B_t}(x)}{-\log\mu(B_t)}\geq 1
%\end{equation}
For any $0<\delta<\frac12$ we define the sets
$$\mathcal{L}_{\delta}^i:=\{x\in\mathcal{X}\ |\ \liminf_{t\to\infty}\frac{\log \tau_{B_t}^i(x)}{-\log \mu(B_t)}\leq 1-2\delta\},$$
$$\mathcal{U}_{\delta}^{i}:=\{x\in\mathcal{X}\ |\ \limsup_{t\to\infty}\frac{\log \tau_{B_t}^i(x)}{-\log \mu(B_t)}\geq 1+2\delta\},$$
and note that the condition $\mu(\mathcal{U}_{\delta}^i)=0$ for all $0<\delta<\frac12$ implies that for a.e. $x\in \cX$
\begin{equation}\label{upper.bound}
\limsup\limits_{t\to\infty}\frac{\log \tau^{i}_{B_t}(x)}{-\log\mu(B_t)}\leq 1
\end{equation}
and similarly the condition that $\mu(\mathcal{L}_{\delta}^i)=0$ for all $0<\delta<\frac12$ implies that for a.e. $x\in \cX$
\begin{equation}\label{lower.bound}
\liminf\limits_{t\to \infty}\frac{\log \tau^{i}_{B_t}(x)}{-\log\mu(B_t)}\geq 1.
\end{equation}

%, it suffices to show $\mu(\mathcal{L}_{\delta}^i)=\mu(\mathcal{U}_{\delta}^{i})=0$ for any $0<\delta<\frac12$.
% since this would imply
%\begin{equation}\label{lre}
%\mu\left(\left\{x\in\mathcal{X}\ |\ \liminf\limits_{t\to\infty}\frac{\log \tau^i_{B_t}(x)}{-\log\mu(B_t)}< 1\right\}\right)=\lim\limits_{\delta\to 0^+}\mu(\mathcal{L}^i_{\delta})=0
%\end{equation}
%and
%\begin{equation}\label{ure}
%\mu\left(\left\{x\in\mathcal{X}\ |\ \limsup\limits_{t\to\infty}\frac{\log \tau^i_{B_t}(x)}{-\log\mu(B_t)}> 1\right\}\right)=\lim\limits_{\delta\to 0^+}\mu(\mathcal{U}^i_{\delta})=0.
%\end{equation}

Now, for any integer $m\geq 1$ and measurable subset $B\subset \mathcal{X}$, define the hitting set
\begin{equation}\label{e:hitting1}\mathcal{H}_{m,B}^i:=\left\{x\in\mathcal{X}\ |\ \tau_B^i(x)\leq m \right\}
\end{equation}
and its complement
\begin{equation}\label{e:missing}\mathcal{M}^i_{m,B}:=\left\{x\in \mathcal{X}\ |\ \tau_B^i(x)> m\right\}.\end{equation}
Note that %$\mathcal{H}^i_{m,B}$ and $\mathcal{M}^i_{m,B}$ are complement to each other and
$x\in \mathcal{H}^i_{m,B}$ (resp. $x\in\mathcal{M}^i_{m,B}$) means the first $m$ steps in the orbit of $x$ hit the set $B$ at least (resp. strictly less than) $i$ times. For any $0<\delta<\frac12$ and any $t>0$ let
$$m^{\pm}_{\delta}(t)=\floor{\frac{1}{\mu(B_t)^{(1\pm\delta)}}}.$$
If $x\in \mathcal{L}_{\delta}^i$, then there exists an unbounded sequence of $t$ such that $\frac{\log\tau_{B_t}^i(x)}{-\log\mu(B_t)}< 1-\delta$, or equivalently, $\tau_{B_t}^i(x)<\frac{1}{\mu(B_t)^{(1-\delta)}}$. Since $\tau_{B_t}^i(x)$ is integer-valued, this implies that $\tau_{B_t}^i(x)\leq \floor{\frac{1}{\mu(B_t)^{(1-\delta)}}}$ for unbounded values of $t$. Hence, $x\in \mathcal{L}_{\delta}^i$ implies that $x\in \mathcal{H}_{m^{-}_{\delta}(t),B_t}^i$ for unbounded values of $t$. Similarly, $x\in \mathcal{U}_{\delta}^i$ implies that $x\in \mathcal{M}^i_{m^{+}_{\delta}(t), B_t}$ for unbounded values of $t$. Let $\mathcal{N}$ denote the set of integers $\ell\geq 0$ such that $\{\mu(B_t)\ |\ t> 0\}\cap [\frac{1}{2^{\ell}}, \frac{1}{2^{\ell+1}})$ is nonempty.
%there exists $t>0$ with $\mu(B_t)\in[\frac{1}{2^{\ell+1}},\frac{1}{2^{\ell}})$ and note
Note that $\mathcal{N}$ is unbounded since $\lim\limits_{t\to\infty}\mu(B_t)= 0$. Thus
$$\mathcal{L}_{\delta}^i\subset \bigcap_{m=0}^{\infty}\bigcup_{\substack{\ell\geq m\\\ell\in\mathcal{N}}}\bigcup_{\frac{1}{2^{\ell+1}}\leq \mu(B_t)< \frac{1}{2^{\ell}}}\mathcal{H}_{m^-_{\delta}(t),B_t}^i$$
and
$$\mathcal{U}_{\delta}^i\subset \bigcap_{m=0}^{\infty}\bigcup_{\substack{\ell\geq m\\\ell\in\mathcal{N}}}\bigcup_{\frac{1}{2^{\ell+1}}\leq \mu(B_t)< \frac{1}{2^{\ell}}}\mathcal{M}_{m^+_{\delta}(t),B_t}^i.$$
For each $\ell\in\mathcal{N}$, let
$$\overline{B}_{\ell}:= \bigcup_{\frac{1}{2^{\ell+1}}\leq \mu(B_t)< \frac{1}{2^{\ell}}} B_t\quad\textrm{and}\quad \underline{B}_{\ell}:= \bigcap_{\frac{1}{2^{\ell+1}}\leq \mu(B_t)< \frac{1}{2^{\ell}}} B_t.$$
Since $\{B_t\}_{t>0}$ is monotone,
$$\frac{1}{2^{\ell+1}}\leq \mu\left(\underline{B}_{\ell}\right)\leq \mu\left(\overline{B}_{\ell}\right)\leq \frac{1}{2^{\ell}}.$$
Moreover, for any $t$ such that $\mu(B_t)\in [\frac{1}{2^{\ell+1}},\frac{1}{2^{\ell}})$,
$$\floor{2^{\ell(1\pm\delta)}}< m^{\pm}_{\delta}(t)\leq \floor{2^{(\ell+1)(1\pm\delta)}}.$$
By construction, for any $m\leq m'$ and $B\subset B'$, $\mathcal{H}_{m,B}^i\subset \mathcal{H}_{m',B'}^i$ and $\mathcal{M}_{m,B}^i\supset \mathcal{M}_{m',B'}^i$. Hence for any $\ell\in\mathcal{N}$, we have
$$\bigcup_{\frac{1}{2^{\ell+1}}\leq \mu(B_t)< \frac{1}{2^{\ell}}}\mathcal{H}_{m^-_{\delta}(t),B_t}^i\subset \mathcal{H}^i_{\floor{2^{(\ell+1)(1-\delta)}}, \overline{B}_{\ell}}\quad\textrm{and}\quad \bigcup_{\frac{1}{2^{\ell+1}}\leq \mu(B_t)< \frac{1}{2^{\ell}}}\mathcal{M}_{m^+_{\delta}(t),B_t}^i\subset \mathcal{M}^i_{\floor{2^{\ell(1+\delta)}},\underline{B}_{\ell}}.$$

Combining the  above arguments gives the following:
\begin{Lem}\label{ght}
For a monotone family, $\{B_t\}_{t>0}$, of shrinking targets in $\mathcal{X}$. If for all sufficiently small $\delta>0$
\begin{equation}\label{lb}
\sum_{\ell\in\mathcal{N}}\mu\left(\mathcal{H}^i_{\floor{2^{(\ell+1)(1-\delta)}}, \overline{B}_\ell}\right)< \infty,
\end{equation}
then the lower bound $(\ref{lower.bound})$ holds for a.e. $x\in\mathcal{X}$. Similarly, if for all sufficiently small $\delta>0$
\begin{equation}\label{up}
\sum_{\ell\in\mathcal{N}}\mu\left(\mathcal{M}^i_{\floor{2^{\ell(1+\delta)}},\underline{B}_\ell}\right)<\infty,
\end{equation}
then the upper bound $(\ref{upper.bound})$ holds for a.e. $x\in\mathcal{X}$.
\end{Lem}

In the following sections we shall show that \eqref{up} holds for one-parameter flows on homogenous spaces. The condition  $\eqref{lb}$, on the other hand, holds in general without any extra assumption on the flow or the shrinking targets. Hence, the following lower bound holds in general.
\begin{Lem}
\label{lower bound} Let $\{B_t\}_{t>0}$ be a monotone family of shrinking targets in $\mathcal{X}$. Then $(\ref{lb})$ holds for all $0< \delta< \frac12$. In particular, for a.e. $x\in \mathcal{X}$
$$\liminf\limits_{t\to \infty}\frac{\log \tau^i_{B_t}(x)}{-\log\mu(B_t)}\geq 1.$$
\end{Lem}
\begin{proof}
For any integer $m\geq 1$ and any measurable set $B\subset \mathcal{X}$, we first show the trivial estimate
 $\mu\left(\mathcal{H}_{m,B}^i\right)\leq  m\mu(B)$. By definition $\mathcal{H}_{m,B}^i= \bigcup_{k=i}^m\left\{x\in\mathcal{X}\ |\ \tau_B^i(x)=k\right\}$ and by minimality $\left\{x\in\mathcal{X}\ |\ \tau_B^i(x)=k\right\}\subset \left\{x\in\mathcal{X}\ |\ xh_k\in B\right\}=Bh_{-k}$. Hence indeed  $\mu\left(\mathcal{H}_{m,B}^i\right)\leq m\mu(B)$. For each $\ell\in\mathcal{N}$ applying this estimate to $\mathcal{H}^i_{\floor{2^{(\ell+1)(1-\delta)}}, \overline{B}_{\ell}}$ we get
$$\mu\left(\mathcal{H}^i_{\floor{2^{(\ell+1)(1-\delta)}}, \overline{B}_{\ell}}\right)\leq \floor{2^{(\ell+1)(1-\delta)}}\mu(\overline{B}_{\ell})\ll 2^{-\delta\ell},$$
where for the last inequality we used $\mu(\overline{B}_{\ell})\leq \frac{1}{2^{\ell}}$.
Hence
\begin{displaymath}
\sum_{\ell\in\mathcal{N}}\mu\left(\mathcal{H}^i_{\floor{2^{(\ell+1)(1-\delta)}}, \overline{B}_{\ell}}\right)\ll \sum_{\ell=0}^{\infty} 2^{-\delta\ell}<\infty.\qedhere
\end{displaymath}
\end{proof}

\subsection{Orbits eventually always hitting}
Given a monotone sequence of shrinking targets $\left\{B_m\right\}_{m\in\N}$, we defined the eventually always hitting set to be
$$\mathcal{A}_{\mathrm{ah}}=\left\{x\in \mathcal{X}\ |\ \textrm{$xH_m^+\cap B_m\neq \emptyset$ for all sufficiently large $m$}\right\}.$$
%We also define for any target $B\subseteq \cX$ the set of points with missing orbits
%\begin{equation}\label{e:CTBo}
%\cC_{m,B}^o=\{x\in \cX: xH_m^+\cap B=\emptyset\}. \quad\textrm{\color{red}{$\cC_{m,B}^o$ coincides with previous $\mathcal{M}_{m,B}^1$}}
%\end{equation}
In  \cite[Proposition 12 and Lemmas 13,14] {Kelmer17b} the first author gave sufficient conditions implying that $\cA_{\mathrm{ah}}$ is a null or co-null set. For the readers convenience we summarize these results in the following:
\begin{Lem}\label{l:EventuallyHitting}
Given a measure preserving ergodic $\bbZ$-action of a group $H$ on a probability space $(\cX,\mu)$, and a monotone sequence, $\{B_m\}_{m\in \N}$,  of shrinking targets.
\begin{enumerate}
\item
If along some subsequence, we have that $m_j\mu(B_{m_j})\leq c<1$, then $\mu(\cA_{\rm ah})=0$.
\item If
$\sum_j \mu(\mathcal{M}^1_{2^{j-1},B_{2^{j}}})<\infty$ then $\mu(\cA_{\rm ah})=1$. If in addition also  $\mu(B_{2^j})\asymp \mu(B_{2^{j+1}})$ and $\sum_j \mu(\mathcal{M}^1_{2^{j+1},B_{2^{j}}})<\infty$ then for a.e. $x\in \cX$, for all sufficiently large $m$
$$\#(xH^+_m\cap B_m)\asymp m\mu(B_m).$$
\end{enumerate}
\end{Lem}

The results of \cite{Kelmer17b} were given for a more general setting of $\bbZ^d$-actions. For $\bbZ$-actions we also have the following lemma stating that $\mathcal{A}_{\mathrm{ah}}$ is always either null or co-null.
\begin{Lem}
\label{zol}
Let $H$ be a measure preserving ergodic $\bbZ$-action on a probability space $(\cX,\mu)$ and let $\{B_m\}_{m\in \N}$ denote a monotone sequence of shrinking targets.
Then $\mathcal{A}_{\mathrm{ah}}$ has measure either zero or one.
\end{Lem}

\begin{proof}
Suppose $\mu(\mathcal{A}_{\mathrm{ah}})>0$, we want to show that $\mu(\mathcal{A}_{\mathrm{ah}})=1$.
Define the set
$$\mathcal{A}':=\{x\in\mathcal{A}_{\mathrm{ah}}\ |\ xh_k\in \mathcal{A}_{\mathrm{ah}}\ \textrm{for all $k\geq 1$}\}.$$
It is clear that $\mathcal{A}'h_1\subset \mathcal{A}'$. Hence by ergodicity it suffices to show that $\mu(\mathcal{A}')>0$. To show this, we prove that $\mu(\mathcal{A}_{\mathrm{ah}}\backslash\mathcal{A}')=0$. Note that
$$\mathcal{A}_{\mathrm{ah}}\backslash\mathcal{A}'=\{x\in\mathcal{A}_{\mathrm{ah}}\ |\ xh_k\notin\mathcal{A}_{\mathrm{ah}}\ \textrm{for some $k\geq 1$}\}=\bigcup_{k=1}^{\infty}\{x\in\mathcal{A}_{\mathrm{ah}}\ |\ xh_k\notin \mathcal{A}_{\mathrm{ah}}\}.$$
Hence, it suffices to show that for any fixed $k$, the set $\{x\in\mathcal{A}_{\mathrm{ah}} |\ xh_k\notin \mathcal{A}_{\mathrm{ah}}\}$ has measure zero.
We note that $x\in\mathcal{A}_{\mathrm{ah}}$ means that for all sufficiently large $m$, there exists some $1\leq j(m)\leq m$ such that $xh_{j(m)}\in B_m$. If for all sufficiently large $m$ we can take $j(m)> k$, then we have $xh_kh_{j(m)-k}\in B_m$ with $1\leq j(m)-k< m$, thus $xh_k\in \mathcal{A}_{\mathrm{ah}}$. Hence, if $x\in\mathcal{A}_{\mathrm{ah}}$ but $xh_k\notin\mathcal{A}_{\mathrm{ah}}$, then there exist infinitely many values of $m$ for which there exists some $1\leq j(m)\leq k$ such that $xh_{j(m)}\in B_m$ but $xh_i\notin B_m$ for all $k< i\leq m$. In particular, for such $x$, there exists some $1\leq j\leq k$ such that $xh_j\in B_m$ for infinitely many values of $m$. Since $\{B_m\}_{m\in\N}$ is monotone, this implies that $xh_j\in \bigcap_{m=1}^{\infty}B_m$. Hence the set
$$\left\{x\in\mathcal{A}_{\mathrm{ah}}\ |\ xh_k\notin \mathcal{A}_{\mathrm{ah}}\right\}\subset \bigcup_{j=1}^{k}\left(\bigcap_{m=1}^{\infty}B_m\right)h_{-j},$$
is of measure zero.
\end{proof}

\subsection{Dynamical Borel-Cantelli and Quasi-independence}
The second part of the classical Borel-Cantelli lemma requires pairwise independence.
The following argument,  going back to Schmidt, shows that a weaker condition of quasi-independence is enough.
Explicitly,  let $\mathcal{F}=\{f_m\}_{m\in\mathbb{N}}$ denote a sequence of functions on the probability space $(\mathcal{X},\mu)$ taking values in $[0,1]$. For $m\in \mathbb{N}$ let $E_m^{\mathcal{F}}= \sum_{1\leq j\leq m}\mu(f_j)$ and $S_m^{\mathcal{F}}(x)= \sum_{1\leq j\leq m}f_j(x)$. We then have:
\begin{Lem}{\cite[Chapter I, Lemma 10]{Sp79}}
\label{gbc}Assuming that for some constant $C>0$, for all $m,n\in \mathbb{N}$
\begin{equation}
\label{keybound}\int_{\mathcal{X}}\left(\sum_{i=m}^nf_i(x)-\sum_{i=m}^n\mu(f_i)\right)^2d\mu(x)\leq C\sum_{i=m}^n\mu(f_i),
\end{equation}
then for any $\epsilon>0$ for a.e. $x\in\mathcal{X}$
$$S_m^{\mathcal{F}}(x)= E_m^{\mathcal{F}}+ O_{\epsilon}\bigg(\sqrt{E_m^{\mathcal{F}}}\log^{\frac32+\epsilon}(E_m^{\mathcal{F}})\bigg).$$
In particular,  if $\left\{E_m^{\mathcal{F}}\right\}_{m\in\mathbb{N}}$ is unbounded, then for a.e. $x\in\mathcal{X}$
$$\lim_{m\to\infty}\frac{S_{m}^{\mathcal{F}}(x)}{E_m^{\mathcal{F}}}= 1.$$
\end{Lem}
Given a  $\bbZ$-action of a group $H=\{h_m\}_{m\in \bbZ}$ on $\cX$, and a sequence of targets, $\{B_m\}_{m\in\mathbb{N}}$, let $f_j(x)=\chi_{B_j}(xh_j)$ so that $S_m^{\mathcal{F}}(x)=\#\{1\leq j\leq m\ |\ xh_j\in B_j\}$. If the events, $xh_j\in B_j$, were pairwise independent, then the left hand side of \eqref{keybound} would be zero. This result shows that the weaker quasi-independence bound \eqref{keybound},  is enough to show that the sequence $\{B_m\}_{m\in \N}$ is sBC for the flow.

\section{Decay of matrix coefficients}
We now collect the needed results on the decay of matrix coefficients for representations of semisimple Lie groups, and use them to prove Theorem \ref{t:edecay}, giving a uniform bound for decay of matrix coefficients for all one-parameter flows, as well as Theorem \ref{t:csd} classifying SD flows. We first give some background on semisimple Lie groups, their Lie algebras and restricted root systems.
We then focus on simple groups and treat the cases of simple groups of real rank one and simple groups of higher rank separately. Finally, we combine our results on simple groups to handle the general case of semisimple groups with property $(T)$.

\subsection{Preliminaries on Lie groups}\label{s:pLie}
Let $G$ be a connected semisimple Lie group with finite center and no compact factors,
%%$connected compoent of the \bbR$-points of a connected semisimple linear algebraic group defined over $\bbR$%%
and let  $\fg$ denote its Lie algebra. Fix a Cartan involution, $\theta$, on $\fg$ and let
$$\fg=\mathfrak{k}\oplus \mathfrak{p}$$
denote the corresponding Cartan decomposition, where $\mathfrak{k}$ is the $+1$ eigenspace and $\mathfrak{p}$ is the $-1$ eigenspace of $\theta$. Let $K\leq G$ denote the maximal compact subgroup with Lie algebra $\fk$. Let $\fa$ be a maximal abelian subspace of $\mathfrak{p}$, and $\mathfrak{m}$ the centralizer of $\fa$ in $\mathfrak{k}$. Denote by $\Phi=\Phi_{\R}(\fa,\fg)$ the set of restricted roots with respect to the pair $(\fa, \fg)$. Fix a set of simple roots $\Delta=\left\{\alpha_i\ |\ i\in I\right\}$ and let $\Phi^+=\Phi_{\R}^+(\fa,\fg)$ be the corresponding set of positive roots.
Then $\mathfrak{g}$ has a root-space decomposition
$$\mathfrak{g}=\mathfrak{m}\oplus\mathfrak{a}\bigoplus_{\lambda\in\Phi_{\R}(\mathfrak{a},\mathfrak{g})}\mathfrak{g}_{\lambda}$$ and an Iwasawa decomposition
$$\mathfrak{g}=\mathfrak{n}\oplus\mathfrak{a}\oplus\mathfrak{k},$$
where $\mathfrak{n}=\bigoplus_{\lambda\in\Phi_{\R}^{+}(\mathfrak{a},\mathfrak{g})}\mathfrak{g_{\lambda}}$. Let $G=NAK$ be the corresponding Iwasawa decomposition of $G$. Let $\fa^{+}=\left\{X\in\fa\ |\ \alpha(X)\geq 0\ \textrm{for all $\alpha\in \Delta$}\right\}$ be the positive Weyl chamber determined by $\Delta$, and $A^{+}$ the corresponding positive Weyl chamber in $A$. The choice of $\Delta$ determines a partial order on $\Phi^+_{\R}(\fa,\fg)$ in the sense that $\lambda\geq \lambda'$ if and only if $\lambda(X)\geq \lambda'(X)$ for any $X\in \fa^+$, and we can fix a total order
%Do we have such an ordering only for simple groups or also for semisimple?
\begin{equation}\label{e:order}
\lambda_1\geq \lambda_2\geq \cdots\geq \lambda_L
\end{equation}
 on $\Phi^+_{\R}(\fa,\fg)$ that is compatible with this partial order, where $L=|\Phi_{\R}^+(\fa,\fg)|$ (see \cite[p. 155]{Knapp02} for an example of such an order). Let $d_i=\mathrm{dim}( \mathfrak{g}_{\lambda_i})$ and $d_0=\mathrm{dim} (\mathfrak{m}\oplus \mathfrak{a})$. The following lemma gives a nice matrix representation of the adjoint Lie algebra.

\begin{Lem}{$($cf. \cite[Lemma 6.45]{Knapp02}$)$}
\label{reflemma} There exists a basis of $\mathfrak{g}$ compatible with the above root-space decomposition such that the matrices representing $\textrm{ad}( \mathfrak{g})$ have the following properties:
\begin{enumerate}
\item The matrices of ad $\mathfrak{k}$ are skew symmetric,
\item the matrices of ad $\mathfrak{n}$ are upper triangular with $0$ on the diagonal.
\item the matrices of ad $\mathfrak{a}$ are diagonal with real entries with
\end{enumerate}
$$\textrm{ad}(X)=\begin{pmatrix}
\lambda_1(X)I_{d_1} & & & & & & &\\
 &\ddots& & & & & &\\
 & &\lambda_L(X)I_{d_L}& & & & &\\
 & & & 0I_{d_0}& & & &\\
 & & & & -\lambda_L(X)I_{d_L}& & &\\
 & & & & & \ddots& & \\
 & & & & & & -\lambda_1(X)I_{d_1}
 \end{pmatrix},$$
for $X\in \fa$ where $I_{d_i}$ is the $d_i\times d_i$ identity matrix.
\end{Lem}

As an immediate consequence, we have the following description of $\textrm{Ad}(G)$.
\begin{Cor}
\label{cor 1}There exists a basis of $\mathfrak{g}$ such that the matrices representing $\textrm{Ad} (G)$ have the following properties:
\begin{enumerate}
\item The matrices of $\textrm{Ad} (K)$ are orthogonal,
\item the matrices of $\textrm{Ad} (N)$ are upper triangular with $1$'s on the diagonal,
\item the matrices of $\textrm{Ad} (A)$ are diagonal with real entries with
\end{enumerate}
$$\mathrm{Ad}(\exp(X))=\begin{pmatrix}
e^{\lambda_1(X)}I_{d_1} & & & & & & &\\
 &\ddots& & & & & &\\
 & &e^{\lambda_L(X)}I_{d_L}& & & & &\\
 & & & I_{d_0}& & & &\\
 & & & & e^{-\lambda_L(X)}I_{d_L}& & &\\
 & & & & & \ddots& & \\
 & & & & & & e^{-\lambda_1(X)}I_{d_1}
 \end{pmatrix},$$
for all $X\in \mathfrak{a}$.
\end{Cor}
We will fix such a basis once and for all and use it to identify $\textrm{ad}(\fg)$ and $\textrm{Ad}(G)$ with the corresponding matrix groups. For future reference, we prove the following short lemma.
\begin{Lem}\label{Lem 2}\label{l:nilpotnent}
For any nonzero $X\in\fg$
%for which $\textrm{ad}(X)$ is nilpotent,
we have that  $\left(\textrm{ad}(X)\right)^2\neq 0$.
\end{Lem}
\begin{proof}
It suffices to show that $\textrm{ad}(X)\circ\textrm{ad}(X)$ is nontrivial as an endomorphism of $\fg$. If $\textrm{ad}(X)$ is not nilpotent there is nothing to show. If it is nilpotent,  by \cite[Theorem 7.4, p. 432]{Helgason78} there exist elements $H, Y\in\fg$ such that
$$[H,X]=2X,\qquad [H,Y]=-2Y,\qquad [X,Y]=H.$$
Applying $\textrm{ad}(X)\circ\textrm{ad}(X)$ to $Y$ we get $$\left(\textrm{ad}(X)\circ\textrm{ad}(X)\right)(Y)= \left[X,[X,Y]\right]=[X,H]= -2X\neq 0,$$
thus completing the proof.
\end{proof}

\subsection{Cartan decomposition of one-parameter subgroups}
We now consider the case where $G$ a simple, that is $\fg$ is simple as a real Lie algebra, and study the Cartan decomposition of a one-parameter subgroup $\{h_t=\exp(tX_0)\}$, with $X_0\in\fg$.

Fix an order $\lambda_1\geq \lambda_2\geq \ldots\geq \lambda_L$ on the set of positive roots $\Phi^+$ as in \eqref{e:order}. Since $G$ is simple, $\Phi$ is irreducible and the root $\lambda_1\in \Phi^+$, is the highest root, characterized by the property that $\lambda_1(X)\geq \lambda(X)$ for any $\lambda\in \Phi^+$ and $X\in \fa^+$.  Recalling the Cartan decomposition $G=KA^+K$, we can write
\begin{equation}
\label{cartan}h_t= k_1(t)\exp\left(X(t)\right)k_2(t)
\end{equation}
with $k_i(t)\in K$ and $X(t)\in \fa^{+}$, and note that $X(t)\in\fa^+$ is uniquely determined by $(\ref{cartan})$. Now, by the complete additive Jordan decomposition, there exist three pairwise commuting elements $X_{\fn}$, $X_{\fa}$ and $X_{\fk}$ in $\fg$ such that $X_0=X_{\fn}+X_{\fa}+X_{\fk}$ with $\textrm{ad}(X_{\fn}), \textrm{ad}(X_{\fa})$ and $\textrm{ad}(X_{\fk})$ $\textrm{Ad}(G)$-conjugate to elements in $\textrm{ad}(\fn), \textrm{ad}(\fa)$ and $\textrm{ad}(\fk)$ respectively.
%Note that the flow being unbounded means exactly that $X_{\fn}+X_{\fa}\neq 0$.
We say a one-parameter subgroup is \emph{quasi-diagonalizable} if $X_{\fa}\neq 0$, and that is \emph{quasi-unipotent} if $X_{\fa}=0$ and $X_{\fn}\neq 0$ (note that any unbounded subgroup is either quasi-diagonalizable or quasi-unipotent). We first prove the following:

\begin{Prop}\label{mprop}
Given an unbounded one-parameter subgroup $h_t= k_1(t)\exp\left(X(t)\right)k_2(t)
$, if it is quasi-diagonalizable then
there exists some constant $c>0$ such that
$$e^{\lambda_1(X(t))}\gg e^{ct},$$
and if it is quasi-unipotent then there exists some integer $l\geq 2$ such that
$$e^{\lambda_1(X(t))}\asymp |t|^l.$$
%where the implied constants may depend on the flow but not on $t$.
\end{Prop}
\begin{proof}
For any $g\in G$, let $\|g\|^2:=\textrm{tr}\left(\textrm{Ad}(g)^t\textrm{Ad}(g)\right)$. That is, $\|g\|^2$ is the Hilbert-Schmidt norm of $\textrm{Ad}(g)$, and equals the sum of squares of entries of $\textrm{Ad}(g)$.
By Cauchy-Schwartz, $\|g_1g_2\|\leq \|g_1\|\|g_2\|$ for any $g_1, g_2\in G$, so in particular, for any fixed $g_0\in G$ and all $h\in G$ we have
\begin{equation}\label{hsb}
\|g_0^{-1}hg_0\|\asymp_{g_0}\|h\|.
\end{equation}

Now, on one hand, in view of $(\ref{cartan})$ and Corollary \ref{cor 1},
$$\|h_t\|^2=\textrm{tr}\left(\textrm{Ad}\left(\exp(X(t))\right)^2\right)=\sum_{i=1}^L d_i\left(e^{2\lambda_i\left(X(t)\right)}+ e^{-2\lambda_i\left(X(t)\right)}\right)+d_0,$$
and since $X(t)\in \fa^+$ we have that $\lambda_1\left(X(t)\right)\geq \lambda_i\left(X(t)\right)\geq 0$ for all  $\lambda_i\in \Phi^+$, hence,
$$\|h_t\|^2\asymp e^{2\lambda_1\left(X(t)\right)}.$$
%where the implied constant only depends on $\dim\fg$.
On the other hand, since $X_{\fn}, X_{\fa}$ and $X_{\fk}$ are pairwise commuting,
$$\|h_t\|=\|\exp(tX_0)\|=\|\exp(tX_{\fn})\exp(tX_{\fa})\exp(tX_{\fk})\|.$$
Moreover, since $\textrm{ad}(X_{\fk})$ is $\textrm{Ad}(G)$-conjugate to some element in $\textrm{ad}(\fk)$, the one-parameter subgroup $\{\exp(tX_{\fk})\}_{t\in\R}$ is compact, and $\{\|\exp(tX_{\fk})\|\}_{t\in \R}$ is uniformly bounded from above (by constants depending only on the generator $X_0$). Hence by Cauchy-Schwartz
$$\|h_t\|\asymp \|\exp(tX_{\fn})\exp(tX_{\fa})\|.$$
If $X_{\fa}$ is nonzero, then there exists some $g\in G$ such that $\textrm{Ad}(g)^{-1}\mathrm{ad}(X_{\fa})\textrm{Ad}(g)=\mathrm{ad}(X_{\fa}')$ with $X_{\fa}'\in \fa$ nonzero. Thus by $(\ref{hsb})$,
$$\|h_t\|^2\asymp \|g^{-1}\exp(tX_{\fn})g\exp(tX_{\fa}')\|^2.$$
Since $\textrm{ad}(X_{\fn})$ is nilpotent, the entries of $\textrm{Ad}(g^{-1}\exp(tX_{\fn})g)$ are all polynomials in $t$ and since it is invertible, each row has at least one nonzero entry. Since $X_{\fa}'\in \fa$, there exists some root $\lambda\in \Phi^+$ such that $|\lambda(X_{\fa}')|=c>0$. In view of Corollary \ref{cor 1}, the matrix $\textrm{Ad}\left(g^{-1}\exp(tX_{\fn})g\exp(tX_{\fa}')\right)$ has at least one entry of the form $e^{ct}P(t)$, where $P(t)$ is some nontrivial polynomial in $t$. Hence in this case, for all sufficiently large $t$,
$$e^{2\lambda_1(X(t))}\asymp \|h_t\|^2\geq e^{2ct}.$$
If $X_{\fa}=0$, then $X_{\fn}$ is nonzero. Let $l$ be the unique integer such that $\left(\textrm{ad}(X_{\fn})\right)^{l}\neq 0$ and $\left(\textrm{ad}(X_{\fn})\right)^{l+1}= 0$ and note that $l\geq 2$ by Lemma \ref{l:nilpotnent}. In this case, the entries of $\textrm{Ad}\left(\exp(tX_{\fn})\right)$ are all polynomials in $t$ with degree less than or equal to $l$, and there exists some entry with degree exactly $l$. Hence in this case, for all sufficiently large $t$
\begin{displaymath}
e^{2\lambda_1(X(t))}\asymp \|h_t\|^2\asymp \|\exp(tX_{\fn})\|^2\asymp t^{2l}.\qedhere
\end{displaymath}
%{\color{red}{Finally, we have $l\geq 2$ by Lemma \ref{Lem 2}.}}
\end{proof}

\subsection{Simple groups of real rank one}
We recall that, up to local isomorphism, a rank one group $G$ is in one of the following four families of groups: $\SO(d,1), \SU(d,1)$, $\Sp(d,1),$ with $d\geq 2$ and $\mathrm{F}_4^{-20}$. In these cases, the positive restricted root system $\Phi_{\R}^{+}(\fa,\fg)$ consists of one or two elements. Let $\alpha$ be the unique element in $\Phi_{\R}^{+}(\fa,\fg)$ such that $\frac12 \alpha\notin \Phi_{\R}^{+}(\fa,\fg)$. Let $p$ be the dimension of $\fg_{\alpha}$, $q$ be the dimension of $\fg_{2\alpha}$ and $\rho=\frac12(p+2q)\alpha$ be half the sum of the positive roots with multiplicities. Explicitly we have that $p, q$ and $\rho$ are as follows:
\begin{displaymath}
\begin{tabular}{| c | c | c | c | c | }
  \hline			
   & $SO(d,1)$ & $SU(d,1)$ & $Sp(d,1)$ & $F_4^{-20}$ \\
  \hline
  $p$ & $d-1$ & $2(d-1)$ &$4(d-1)$ & $8$\\
  \hline
  $q$ & $0$ & $1$ & $3$ & $7$\\
  \hline
  $\rho$ & $\frac{d-1}{2}\alpha$ & $d\alpha$ & $(2d+1)\alpha$ & $11\alpha$\\
  \hline
\end{tabular}
\end{displaymath}
Let $\fa_{\bbC}^{\ast}$ be the complexified dual of $\fa$ and fix $X_1$ to be the unique element in $\fa$ such that $\alpha(X_1)=1$. We identify $\fa_{\bbC}^{\ast}$ with $\bbC$ via their values at $X_1$. Denote by $\widehat{G}$ the unitary dual of $G$ and $\widehat{G}_{K}$ the spherical unitary dual. The spherical unitary dual can be parameterized by $\fa_{\bbC}^{\ast}/W$ where $W$ is the Weyl group. Let $\rho_0\in \fa_{\bbC}^{\ast}$ be defined by
\[\rho_0 = \left\{
  \begin{array}{lr}
    \rho & \textrm{$\fg=\mathfrak{so}(d,1)$ or $\mathfrak{su}(d,1)$}\\
    \rho-2\alpha & \textrm{$\fg= \mathfrak{sp}(d,1)$}\\
    \rho-6\alpha & \textrm{$\fg= \mathfrak{f}_4^{-20}.$}
  \end{array}
\right.
\]
Then with the above identification between $\fa_{\bbC}^{\ast}$ and $\bbC$ we have the parametrization
$$\widehat{G}_{K}=\{\pi_s\ |\ s\in i\R_{\geq0} \cup (0, \rho_0)\}\cup \{\pi_{\rho}\},$$
where the representations $\pi_s, s\in i\R_{\geq 0}$ are the (tempered) principal series representations, the representations $\pi_s, s\in (0, \rho_0)$ are the (non-tempered) complementary series (cf. \cite{Kostant69}), and $\pi_{\rho}$ is the trivial representation. We note that in each representation $\pi_s\in \widehat{G}_{K}$ there is a unique (up to scaling) spherical vector.

For any lattice $\G\leq G$ consider the right regular representation of $G$ on $L^2(\G\bk G)$. We denote by
$L^2_{\mathrm{temp}}(\G\bk G)$ the subspace that weakly contains only tempered representations. We then have a spectral decomposition for any spherical $f\in L^2(\G\bk G)$
\begin{equation}
\label{sdc}f= \langle f,1\rangle +\sum_{k}\langle f,\varphi_k\rangle \varphi_k + f_{\textrm{temp}},
\end{equation}
with $f_{\textrm{temp}}\in L^2_{\textrm{temp}}(\G\bk G)$ and $\varphi_k\in\pi_{s_k}$ with $s_k\in (0,\rho_0)$.

After identifying spherical functions in $L^2(\G\bk G)$ with functions in $L^2(\G\bk \cH)$, the vectors $\varphi_k\in\pi_{s_k}$ occurring in this decomposition are the exceptional Laplacian eigenfunction in $L^2(\G\bk \cH)$, with corresponding eigenvalues $\rho^2-s_k^2$.
In particular, the \emph{spectral gap} for $\G$, i.e., the gap between the trivial eigenvalue and the first non-trivial eigenvalue of the Laplacian on $L^2(\G\bk \cH)$, is  $\tau(2\rho-\tau)$ where the parameter $\tau=\tau(\G)$ is  given by
 \begin{equation}\label{e:SpectralGap}
 \tau(\Gamma):=\min_{k}(\rho-s_k).
 \end{equation}

The (spherical) exceptional forms $\varphi_k\in \pi_{s_k}$ above are either cusp forms (vanishing at all cusps) or residual forms (obtained as residues of Eisenstein series).  The exceptional cusp forms are uniformly bounded, but the residual forms can blow up at the cusps.
To control their growth we recall the following result from  \cite[Lemma 7]{Kelmer17b} on the $L^p$ norms of exceptional forms.\footnote{The result in  \cite{Kelmer17b} is stated for $G=\SO_0(d,1)$ but the proof is identical.}
\begin{Prop}
\label{nE} For any spherical exceptional form $\varphi_k\in \pi_{s_k}$ with $s_k\in(0,\rho_0)$ we have that  $\varphi_k\in L^p(\G\bk G)$ for any $p< \frac{2\rho}{\rho-s_k}$.
\end{Prop}

We now turn to estimate the decay of matrix coefficients for flows on $L^2(\G\bk G)$ for $G$ a rank one group.
For each $s\in i\R_{\geq 0}\cup (0, \rho_0)$ consider the spherical function defined by
$$\phi_s(g)= \langle\pi_s(g)v, v\rangle$$
where $v\in V_{\pi_s}$ is the unique unit spherical vector. For any $\epsilon> 0$, $\phi_s$ decays like
\begin{equation}
\label{pd}|\phi_s(\exp(tX_1))|\ll_{\epsilon} e^{-\rho(1-\epsilon)|t|}
\end{equation}
if $s\in i\R_{\geq 0}$ and
\begin{equation}
\label{cd}|\phi_s(\exp(tX_1))|\ll_{\epsilon} e^{-(\rho-s)(1-\epsilon)|t|}
\end{equation}
if $s\in (0,\rho_0)$. We note that $(\ref{pd})$ follows from the asymptotic behavior of the Harish-Chandra $\Xi$ function (cf. \cite[section $7$]{Howe82}), while for $(\ref{cd})$, we refer to \cite[(5.1.18)]{RV12}. Let $\lambda_1\in\Phi_{\R}^{+}(\fa,\fg)$ be the highest root as before, and define $\kappa=\kappa(G)$ to be the unique integer such that $\alpha=\frac{\kappa}{2}\lambda_1$. Explicitly, $\kappa=2$ if $G$ is locally isomorphic to $SO(d,1)$ and $\kappa=1$ otherwise. Applying $(\ref{pd})$ and $(\ref{cd})$ to the regular representation $\left(\pi,  L^2_0(\G\bk G)\right)$, we get the following:

\begin{Prop}\label{rod}
%For any unbounded one-parameter group  $\{h_t\}_{t\in\R}\leq G$.
Let $G$ be a connected simple Lie group with finite center and real rank one, and let $\Gamma\leq G$ be a lattice in $G$. Let $\{h_t\}_{t\in \R}$ be an unbounded one-parameter subgroup of $G$. For any spherical tempered $\psi,\phi\in L^2_{\mathrm{temp}}(\G\bk G)$, for all $|t|\geq 1$ we have
\begin{equation}
\label{tede}\left|\langle\pi(h_t)\psi,\phi\rangle\right|\ll_{\epsilon} \frac{\|\psi\|_2\|\phi\|_2}{|t|^{\kappa\rho(1-\epsilon)}}.
\end{equation}
For any spherical non-tempered exceptional form $\varphi\in \pi_s$ with $s\in (0,\rho_0)$, for all $|t|\geq 1$ we have
\begin{equation}
\label{comde}|\langle\pi(h_t)\varphi,\varphi\rangle|\ll_{\epsilon}\frac{\|\varphi\|_2^2}{|t|^{\kappa(\rho-s)(1-\epsilon)}}.
\end{equation}
\end{Prop}

\begin{proof}
Recall the Cartan decomposition $h_t=k_1(t)\exp(X(t))k_2(t)$ with $k_i(t)\in K$ and $X(t)\in \fa^{+}$. In view of $(\ref{pd})$ and $(\ref{cd})$, it suffices to show that $e^{\alpha(X(t))}\gg |t|^{\kappa}$. Recall that $\kappa$ is defined precisely such that $\alpha=\frac{\kappa}{2}\lambda_1$ where $\lambda_1$ is the highest root. By Proposition \ref{mprop}, $e^{\lambda_1(X(t))}\gg t^2$, hence $e^{\alpha(X(t))}\gg (t^2)^{\frac{\kappa}{2}}=|t|^{\kappa}$.
\end{proof}
Using Proposition \ref{rod} together with the spectral decomposition, we see that the rate of decay is controlled by the spectral gap parameter $\tau(\G)$ as follows:
\begin{Thm}
For $G$ and $\G$ as in Proposition \ref{rod}, for all sufficiently small $\epsilon>0$, for any spherical $f_1,f_2\in L^2_0(\G\bk G)$ and for all $|t|\geq 1$, we have
$$\left|\langle \pi(h_t)f_1,f_2\rangle\right|\ll_{\epsilon}\frac{\|f_1\|_2\|f_2\|_2}{|t|^{\kappa\tau(\Gamma)(1-\epsilon)}}.$$
\end{Thm}

\begin{proof}
For any spherical $f_1,f_2\in L^2_0(\G\bk G)$, applying the spectral decomposition $(\ref{sdc})$ and Proposition \ref{rod} we get for any $|t|\geq 1$
\begin{align*}
|\langle \pi(h_t)f_1, f_2\rangle|&\leq\sum_{k}\left|\langle f_1, \varphi_k\rangle \overline{\langle f_2, \varphi_k\rangle}\right|\left| \langle \pi(h_t)\varphi_k, \varphi_k\rangle\right| +\left|\langle \pi(h_t)f^1_{\textrm{temp}}, f^2_{\textrm{temp}}\rangle\right|\\
&\ll_{\epsilon} \sum_k \frac{\left|\langle f_1, \varphi_k\rangle \overline{\langle f_2, \varphi_k\rangle}\right|}{|t|^{\kappa(\rho-s_k)(1-\epsilon)}}+ \frac{\|f^1_{\textrm{temp}}\|_2\|f^2_{\textrm{temp}}\|_2}{|t|^{\kappa\rho(1-\epsilon)}}\\
&\leq  \frac{\sum_k\left|\langle f_1, \varphi_k\rangle \overline{\langle f_2, \varphi_k\rangle}\right|+\|f^1_{\textrm{temp}}\|_2\|f^2_{\textrm{temp}}\|_2}
{|t|^{\kappa\tau(\Gamma)(1-\epsilon)}}\\
&\leq \frac{\|f_1\|_2\|f_2\|_2}{|t|^{\kappa\tau(\Gamma)(1-\epsilon)}},
\end{align*}
where for the last inequality we used Cauchy-Schwartz.
\end{proof}
As a direct corollary, we have the following:
\begin{Cor}
\label{uod}
For $G$ a connected simple Lie group with finite center and real rank one, not locally isomorphic to $SO(2,1)$ and $\Gamma\leq G$ a lattice in $G$. If $\tau(\Gamma)> \frac{1}{\kappa(G)}$, then any unbounded one-parameter flow is SD for $\G\bk G$. In particular, if $G$ is locally isomorphic to $\Sp(d,1)$ with $d\geq 2$ or $\mathrm{F}_4^{-20}$, then any unbounded one-parameter flow is SD for $\G\bk G$ for any lattice $\G$.
\end{Cor}

\begin{Rem}
For $G$ is locally isomorphic to $\SO(d+1,1)$ and $\G$ a congruence lattice, the best known bounds towards the Selberg-Ramanujan conjecture  \cite{BurgerSarnak1991,KimSarnak03,BlomerBrumley11,BergeronClozel13,BergeronClozel17} imply that $\tau(\G)\geq \tfrac{25}{64}$ for $d=1$, that $\tau(\G)\geq \tfrac{25}{32}$ for $d=2$ and that $\tau(\G)\geq 1$ when $d\geq 3$. Similarly, when $G$ is locally isomorphic to $\SU(d,1)$, we have that $\tau(\G)\geq \tfrac{6}{5}$ when $d=2$ and that $\tau(\G)\geq 2$ for $d\geq 3$.
In particular, it follows that for any rank one group, $G$, not locally isomorphic to $\SL_2(\R)$, for any congruence lattice $\G\leq G$, we have that  $\kappa \tau(\G)>1$ and all unbounded one-parameter flows on $\G\bk G$ are SD.
\end{Rem}
%\footnotetext{\color{red}{Is it $"\geq"$ or $">"$? If $\tau(\Gamma)= 1$ then our result can not imply that any unbounded one-parameter flow is SD for $\G\bk SU(d,1)$.}}

\subsection{Simple groups of higher rank}
Next, we consider the case where  $G$ is a connected simple Lie group with finite center and real rank $\geq 2$.
Following  \cite{Oh98}, two roots $\alpha$ and $\beta$ are called strongly orthogonal if neither one of $\alpha\pm\beta$ is a root. Let $S(\Phi)$ denote the family of all subsets of $\Phi^{+}$ whose elements are pairwise strongly orthogonal. We call an element $\mathcal{O}$ in $S(\Phi)$ a strongly orthogonal system. Let $\varrho$ be the function on $S(\Phi)$ given by $\varrho(\mathcal{O})=\sum_{\alpha\in\mathcal{O}}\alpha$.

\begin{Prop}\cite[Proposition 2.3]{Oh98}
\label{so} Let $\Delta=\{\alpha_i\ |\ i\in I\}$ be the set of simple roots as above. There exists a maximal strongly orthogonal system $\mathcal{Q}(\Phi)$ in $S(\Phi)$ in the sense that for any $\mathcal{O}\in S(\Phi)$, for each $i\in I$, the coefficient of $\alpha_i$ in $\varrho(\mathcal{Q}(\Phi))$ is greater than or equal to the coefficient of $\alpha_i$ in $\varrho(\mathcal{O})$.
\end{Prop}

Let $\xi=\frac12 \varrho(\mathcal{Q}(\Phi))=\frac12\sum_{\alpha\in\mathcal{Q}(\Phi)}\alpha$, and consider the $K$-bi-invariant function, $F:G\to \R^+$, defined on $A^+$ via
\begin{equation}\label{e:F}
F(\exp(X))=e^{-\xi(X)}.
\end{equation}

We then have the following result from \cite{Oh98}.
\begin{Thm}\label{d.decay}
For $G$ a connected simple group with finite center and real rank $\geq 2$, for any $\epsilon>0$ sufficiently small, for any nontrivial $\sigma\in \widehat{G}_{K}$ and $K$-invariant unit vector $v_\sigma$ of $\sigma$
\begin{equation}
\left|\langle\sigma(g)v_{\sigma},v_{\sigma}\rangle\right|\ll_\epsilon F(g)^{1-\epsilon}\qquad \textrm{for any}\ g\in G.
\end{equation}
%where $F$ is the $K$-bi-invariant function on $G$ with its values $A^{+}$ given by
%$$F(\exp(X))=e^{-\xi(X)}.$$
%with $X\in \fa^+$ and $\xi=\varrho(\mathcal{Q}(\Phi))=\frac12\sum_{\alpha\in\mathcal{Q}(\Phi)}\alpha$ as above.
\end{Thm}
\begin{proof}
When $G$ is a linear group this is \cite[Theorem A]{Oh98}.
%\begin{Rem}\label{removelinear}
In general, by Langlands classification theorem, any $\sigma\in \widehat{G}_{K}$ appears as some quotient of the induced representation $\textrm{Ind}_{NAM}^G\left(\chi\otimes 1_{M}\right)$, where $N, A$ are as above and $M$ is the centralizer of $A$ in $K$, $\chi$ is some character of $A$ and $1_M$ is the trivial representation of $M$. Elements in $\textrm{Ind}_{NAM}^G\left(\chi\otimes 1_{M}\right)$ are measurable functions $f: G\to \mathbb{C}$ satisfying
$$f(namg)= \chi(a)f(g)\ \textrm{for a.e. $g\in G$, with $n\in N, a\in A$ and $m\in M$},$$
and $G$ acts on $\textrm{Ind}_{NAM}^G\left(\chi\otimes 1_{M}\right)$ via the right regular action.
%$$\textrm{Ind}_{NAM}^G\chi\otimes 1_{M}:=\left\{f: G\to \mathbb{C}\ |\ f(namg)= \chi(a)f(g)\ \textrm{for a.e. $g\in G$ with $n\in N, a\in A$ and $m\in M$}\right\},$$
%with the right regular $G$-action,
Since $G$ has finite center, the maximal compact subgroup $K$ contains the center (\cite[Theorem 6.31]{Knapp02}). In particular, $M$ also contains the center. Hence for any $z$ in the center, $z\cdot f(g)= f(gz)= f(zg)= f(g)$. Thus $\sigma\in \widehat{G}_{K}$ descends to an irreducible unitary representation of the adjoint group $\textrm{Ad}(G)$. Identifying $\fg$ with $\textrm{ad}({\fg})$ and applying \cite[Theorem A]{Oh98} to the linear group $\textrm{Ad}(G)$, we get that for any $g\in G$
$$|\langle\sigma(g)v_{\sigma}, v_{\sigma}\rangle|= |\langle\sigma\left(\textrm{Ad}(g)\right)v_{\sigma},v_{\sigma}\rangle|\ll_{\epsilon}F\left(\textrm{Ad}(g)\right)^{1-\epsilon}=F(g)^{1-\epsilon}.$$
%\end{Rem}
\end{proof}

%\footnotetext{We note that the assumption $G$ being linear is not essential here since by the Langlands classification theorem, any irreducible spherical unitary representation of $G$ descends to an irreducible spherical unitary representation of the adjoint group $\textrm{Ad}(G)$.}
 %$\widehat{G}_K= \widehat{\textrm{Ad}(G)}_{\textrm{Ad}(K)}$.}
Applying these results to one-parameter subgroups we get:
\begin{Prop}\label{decay}
Let $G$ denote a connected simple Lie group with finite center and real rank $\geq 2$ and $\{h_t\}_{t\in \R}$ an unbounded one-parameter subgroup.
For any $\epsilon> 0$ sufficiently small, for any nontrivial $\sigma\in\widehat{G}_{K}$ and $|t|\geq 1$,
\begin{equation}\label{u.decay}
 |\langle\sigma(h_t)v_{\sigma}, v_{\sigma}\rangle|\ll_\epsilon |t|^{\epsilon-1}.
\end{equation}
\end{Prop}

\begin{proof}
Since the singleton $\{\lambda_1\}$ constitutes a strongly orthogonal system, Proposition \ref{so}, implies that $\xi(X)\geq \frac12\lambda_1(X)$ for any $X\in\fa^+$, leading to the bound $F(\exp(X))\leq e^{-\lambda_1(X)/2}$.
Now using the Cartan decomposition $h_t=k_1(t)\exp(X(t))k_2(t)$ and Theorem \ref{d.decay} we get that
\begin{equation*}
\left|\langle\sigma(h_t)v_{\sigma},v_{\sigma}\rangle\right|\ll_\epsilon  \exp((\epsilon-1)\xi(X(t)))\leq \exp(\tfrac{(\epsilon-1)}{2}\lambda_1(X(t))).
\end{equation*}
Now, by Proposition \ref{mprop}, we have that  $e^{\lambda_1(X(t))}\gg t^2$, thus concluding the proof.
\end{proof}

\begin{Rem}
\label{fde}In \cite[p. 187-190]{Oh02} both $\xi$ and the highest root $\lambda_1$ are explicitly given in terms of linear combinations of simple roots. By comparing $\xi$ and $\lambda_1$ directly, we note that when $\Phi_{\R}(\fa,\fg)$ is of type $B_n (n\geq 4)$, $D_n (n\geq 4)$, $E_6$, $E_7$, $E_8$ or $F_4$, we have that $\xi(X)\geq \lambda_1(X)$ for any $X\in \fa^{+}$. Hence in these cases we have the following summable decay:
\begin{equation}
\label{fd}|\langle\sigma(h_t)v_{\sigma},v_{\sigma}\rangle|\ll_{\epsilon}|t|^{2(\epsilon-1)}
\end{equation}
for any nontrivial $\sigma\in \widehat{G}_{K}$ and any $|t|\geq 1$. When $\Phi_{\R}(\fa,\fg)$ is not of the above types, if the one-parameter subgroup is quasi-diagonalizable then the matrix coefficients decay exponentially, and if it is quasi-unipotent with $\left(\textrm{ad}(X_{\fn})\right)^3\neq 0$, then we can take $l=3$ in Proposition \ref{mprop} leading to the summable decay:
\begin{equation}\label{fd2}
|\langle\sigma(h_t)v_{\sigma},v_{\sigma}\rangle|\ll_{\epsilon}|t|^{\frac32(\epsilon-1)}
\end{equation}
for any nontrivial $\sigma\in\widehat{G}_K$ and any $|t|\geq 1$.
\end{Rem}

\subsection{Semisimple groups}
We now consider the general case where $G$ is a connected semisimple Lie group with finite center and no compact factors. Let $\Gamma\leq G$ denote an irreducible lattice, and $\{h_t\}_{t\in \R}$ an unbounded one-parameter subgroup of $G$ as before.

To deal with this case, first, note that there is a surjective homomorphism $\prod_{i=1}^mG_i\rightarrow G$ with finite kernel such that each $G_i$ is a noncompact connected simple Lie group with finite center. %Namely, each $G_i$ is either of real rank  $\geq 2$ or locally isomorphic to $Sp(d,1)$ or $F_4^{-20}$.
Let $\tilde{\Gamma}$ be the preimage of $\Gamma$ and $\{\tilde{h}_t\}$ be the identity component of the pre-image of $\{h_t\}$. By replacing $(\Gamma\backslash G, h_t)$ by $(\tilde{\Gamma}\backslash \prod_{i=1}^m G_i, \tilde{h}_t)$ we may assume, without loss of generality, that $G=\prod_{i=1}^m G_i$ and $h_t=(h_t^1,\ldots, h_t^m)$ is unbounded as a one-parameter subgroup of $G$.

Next, our maximal compact subgroup is of the form $K=\prod_{i=1}^mK_i$, with each $K_i$ a maximal compact subgroup of $G_i$. %Following \cite{KleinbockMargulis1996},
By slightly abusing the terminology in \cite{KleinbockMargulis1996}\footnotemark, we say $G_i$ is an \textit{essential factor}
%\footnote
of $G$ if $\{h_t^i\}_{t\in \R}$ is unbounded in $G_i$. After reordering the factors, we can assume that $G_1,\cdots, G_k$ are all the essential factors (since $\{h_t\}$ is unbounded, we have $k\geq 1$). Let $\widehat{G}_{\Gamma}\subset \widehat{G}$ be the set of irreducible unitary representations that are weakly contained in $L^2_0(\G\bk G)$ and $\widehat{G}_{K, \Gamma}=\widehat{G}_{\Gamma}\cap \widehat{G}_{K}$. We first note that for any $\sigma\in \widehat{G}_{K,\Gamma}$, $\sigma$ is of the form $\sigma=\otimes_{i=1}^m \sigma_i$ with each $\sigma_i\in \widehat{G_i}_{K_i}$. Since $\Gamma\leq G$ is irreducible, each $\sigma_i$ is nontrivial.
\footnotetext{In \cite{KleinbockMargulis1996}, $G_i$ is called an essential factor if $\{h_t^i\}_{t\in\mathbb{R}}$ is quasi-diagonalizable, while here we allow $\{h_t^i\}_{t\in\mathbb{R}}$ to be quasi-unipotent.}
%\footnotetext{\color{red}{I checked \cite{KleinbockMargulis1996} and found that their notion of essential factor is different from ours: $G_i$ being an essential factor for them means $h_t^i$ is quasi-diagonalizable, while for us, $h_t^i$ can also be quasi-unipotent, so I just deleted this citation. Or maybe include it and make a remark about the difference?}}
% {\color{red}{(is this true?)}}.
Moreover, a $K$-invariant unit vector $v_{\sigma}$ of $\sigma$ is of the form $v_{\sigma}=\otimes_{i=1}^mv_{\sigma_i}$, where $v_{\sigma_i}$ is the $K_i$-invariant unit vector of $\sigma_i$. Thus
$$|\langle \sigma(h_t)v_{\sigma}, v_{\sigma}\rangle|=\prod_{i=1}^m|\langle \sigma_i(h_t^i)v_{\sigma_i},v_{\sigma_i}\rangle|.$$

First, we consider the case that $G$ (and hence each of its factors) has property $(T)$, that is,  each $G_i$ is either of real rank  $\geq 2$ or locally isomorphic to $Sp(d,1)$ or $F_4^{-20}$.
In this case, for $1\leq i\leq k$ the factor $\{h_t^i\}$ is unbounded and by Proposition \ref{decay} and Corollary \ref{uod}, for any $|t|\geq 1$ we have $|\langle \sigma_i(h_t^i)v_{\sigma_i},v_{\sigma_i}\rangle|\ll_{\epsilon} \frac{1}{|t|^{(1-\epsilon)}}$. While for $k+1\leq i\leq m$, we bound $|\langle \sigma_i(h_t^i)v_{\sigma_i},v_{\sigma_i}\rangle|\leq 1$. We thus have for any $|t|\geq 1$
\begin{equation}\label{e:factor}
|\langle \sigma(h_t)v_{\sigma}, v_{\sigma}\rangle|\ll_{\epsilon} \frac{1}{|t|^{k(1-\epsilon)}}.
\end{equation}
\begin{Prop}\label{p:decay}
Let $G$ be a connected semisimple Lie group with finite center and no compact factors,
$\Gamma\leq G$ an irreducible lattice, and $\{h_t\}_{t\in \R}$ an unbounded one-parameter subgroup of $G$.
If $G$ has property $(T)$, then for any spherical $\varphi, \psi\in L^2_0(\G\bk G)$ and for any $|t|\geq 1$, we have
\begin{equation}\label
{e:decay}
|\langle \pi(h_t)\varphi, \psi\rangle|\ll_\epsilon \frac{\|\varphi\|_2\|\psi\|_2}{|t|^{k(1-\epsilon)}},\end{equation}
where $k\geq 1$ is the number of essential factors of the flow.
\end{Prop}
\begin{proof}
Using the direct integral decomposition of $L^2_0(\G\bk G)$, for any spherical $\varphi\in L^2_0(\G\bk G)$, $\varphi$ can be written as
$$\varphi= \int_{\sigma\in \widehat{G}_{K,\Gamma}}\varphi_{\sigma} d\nu(\sigma),$$
where $\varphi_{\sigma}\in\sigma$ is spherical and $\nu$ is some Borel measure on $\widehat{G}_{\G}$. Hence for any spherical $\varphi, \psi\in  L^2_0(\G\bk G)$, we have
$$\langle \pi(h_t)\varphi, \psi\rangle=\int_{\sigma\in\widehat{G}_{K,\G}}\langle\sigma(h_t)\varphi_{\sigma},\psi_{\sigma}\rangle d\nu(\sigma).$$
By \eqref{e:factor} for any $|t|\geq 1$ we can bound
%%$$|\langle \sigma(h_t)\varphi_{\sigma},\psi_{\sigma}\rangle|\ll_{\epsilon}\frac{||\varphi_{\sigma}||_2||\psi_{\sigma}||_2}{|t|^{k(1-\epsilon)}}.$$
%Hence for any $|t|\geq 1$ we have
\begin{align*}
\left|\langle \pi(h_t)\varphi, \psi\rangle\right|&\leq \int_{\sigma\in \widehat{G}_{K,\Gamma}}\left|\left\langle\sigma(h_t)\varphi_{\sigma},\psi_{\sigma}\right\rangle\right| d\nu(\sigma)\\
&\ll_{\epsilon} \frac{1}{|t|^{k(1-\epsilon)}}\int_{\sigma\in \widehat{G}_{K,\Gamma}}||\varphi_{\sigma}||_2||\psi_{\sigma}||_2 d\nu(\sigma)\\
%&\leq& \frac{1}{|t|^{k(1-\epsilon)}}\left(\int_{\sigma\in \widehat{G}_{K,\Gamma}}||\varphi_{\sigma}||_2^2 d\nu(\sigma)\right)^{\frac12}\left(\int_{\sigma\in \widehat{G}_{K,\Gamma}}||\psi_{\sigma}||_2^2 d\nu(\sigma)\right)^{\frac12}\\
&\leq\frac{\|\varphi\|_2\|\psi\|_2}{|t|^{k(1-\epsilon)}},
\end{align*}
where for the last inequality we used Cauchy-Schwartz.
\end{proof}

\begin{proof}[Proof of Theorem \ref{t:edecay}]
The result follows from Proposition \ref{p:decay} (after noting that for an unbounded flow we have at least one essential factor so $k\geq 1$).
We further note that the only possibility that the flow is not SD is that there is only one essential factor, say $G_1$, and $\{h_t^1\}_{t\in \mathbb{R}}$ (viewed as a one-parameter flow on $\G\bk G$) is not SD.
\end{proof}
\begin{proof}[Proof of Theorem \ref{t:csd}]
The result follows from the above argument and Remark \ref{fde}.
\end{proof}
\begin{Rem}\label{r:conditional2}
For a semisimple group $G$ of real rank $\geq 2$ without property $(T)$ our result is only conditional. In this case, Margulis Arithmeticity Theorem states that any irreducible lattice $\G\leq G$ is arithmetic in the sense that it is commensurable to a congruence lattice defined over some number field. Moreover, Serre's congruence subgroup conjecture (which is settled when $\G\bk G$ is not compact \cite{Rapinchuk92}) states that in fact all irreducible lattices are congruence lattices. Now, when $\G$ is a congruence lattice the generalized Selberg-Ramanujan conjecture \cite{Shahidi04,Sarnak05} gives very precise restrictions on which non-tempered representations may occur in the decomposition of  $L^2(\G\bk G)$, and these representations all have fast decay of matrix coefficients. In particular,  when $\G$ is a congruence lattice, the Selberg-Ramanujan conjecture implies that \eqref{e:decay} holds for all flows on $\G\bk G$, and moreover, when $G$ is not locally isomorphic to a product of copies of $\SL_2(\R)$ and $\SL_2(\bbC)$, this already follows from the known bounds \cite{KimSarnak03, BlomerBrumley11,BergeronClozel13,BergeronClozel17}  towards the Selberg-Ramanujan conjecture.
\end{Rem}

\section{Effective mean ergodic theorems and consequences}
Let $G$ be a connected semisimple Lie group with finite center and no compact factors, $\G\leq G$ an irreducible lattice, $\mathcal{X}=\Gamma\backslash G$, and $H=\{h_m\}_{m\in \bbZ}$ an unbounded discrete time one-parameter flow on $\mathcal{X}$ generated by a one-parameter subgroup as before.
%a discrete unbounded one-parameter subgroup generating a one-parameter flow as before.
For any $f\in L^2(\mathcal{X})$ and any integer $m\geq 1$, define the averaging operator
$$\beta_{m}^+(f)(x):=\frac{1}{m}\sum_{j=1}^{m}f(xh_j).$$
Since $H$ acts ergodically on $\mathcal{X}$, the mean ergodic theorem states that
$$\|\beta_m^+(f)- \mu(f)\|_2\to 0,$$
as $m\to\infty$ for any $f\in L^2(\mathcal{X})$, where $\mu(f):=\int_{\mathcal{X}}fd\mu$. In this section we adapt the method introduced in \cite{GhoshKelmer17} and \cite{Kelmer17b}, to prove two effective mean ergodic theorems using the explicit rate of decay of %correlations
matrix coefficients obtained in the previous section. The arguments are slightly different for rank one groups and for groups with property $(T)$, so we will treat them separately.
%Following \cite{Kelmer17b}, for any integer $m\geq 1$ and $f\in L^2(\mathcal{X},\mu)$, define the set
%$$\mathcal{C}_{m,f}=\left\{x\in \mathcal{X}\ |\ \left|\beta^+_m(f)(x)-\mu(f)\right|\geq \frac{1}{2}\mu(f)\right\}.$$
\subsection{Groups with Property (T)}
When the group $G$ has property $(T)$, we can use the uniform result on decay of matrix coefficients for one-parameter flows to  show the following.
\begin{Prop}
\label{pwt}Assume that $G$ has property (T). Then for any unbounded discrete time one-parameter flow $H=\{h_m\}_{m\in\mathbb{Z}}$, for any $\epsilon> 0$ and for any spherical $f\in L^2(\mathcal{X})$ we have
\begin{equation}
\label{key estimate}\|\beta_m^+(f)-\mu(f)\|_2\ll_{\epsilon}\frac{\|f\|_2}{m^{\frac12(1-\epsilon)}}.
 \end{equation}
 If the flow is SD we have the slightly stronger bound
\begin{equation}
\label{bmet}\|\beta_m^+(f)-\mu(f)\|_2\ll\frac{\|f\|_2}{\sqrt{m}}.
\end{equation}
\end{Prop}
\begin{proof}
Let $f_0=f-\mu(f)\in  L^2_0(\cX)$. Proposition \ref{decay} and Corollary \ref{uod} imply that for any small $\epsilon>0$ for all
$|k|\geq 1$
 $$\left|\<\pi(h_k)f_0,f_0\>\right|\ll_{\epsilon}\frac{\|f_0\|_2^2}{|k|^{1-\epsilon}}.$$
Noting that $\beta_m^+(f_0)= \beta_m^+(f)-\mu(f)$ we have
\begin{align*}
%\|\beta_m^+(f)-\mu(f)\|_2^2&=&
\|\beta_m^+(f_0)\|_2^2
                    &=\frac{1}{m^2}\<\sum_{1\leq i\leq m}\pi(h_i)f_0,\sum_{1\leq j\leq m}\pi(h_j)f_0\>
                    = \frac{1}{m^2}\sum_{1\leq i,j\leq m}\<\pi(h_{i-j})f_0,f_0\>\\
                    &= \frac{1}{m^2}\sum_{|k|\leq m}\<\pi(h_k)f_0,f_0\>\#\{(i,j)\ |\ 1\leq i,j\leq m, i-j=k\}\\
                    &\leq \frac{1}{m}\sum_{|k|\leq m}|\<\pi(h_k)f_0,f_0\>|
                     \ll_{\epsilon} \frac{1}{m}(1+2\sum_{k=1}^{m}\frac{1}{k^{1-\epsilon}})\|f_0\|_2^2
                   \ll\frac{\|f_0\|_2^2}{m^{1-\epsilon}}\leq \frac{\|f\|_2^2}{m^{1-\epsilon}}.
\end{align*}
If the flow is SD the same argument with the stronger bound $\left|\<\pi(h_k)f_0,f_0\>\right|\ll_{\eta}\frac{\|f_0\|_2^2}{|k|^{\eta}},$
gives $\|\beta_m^+(f)-\mu(f)\|_2^2\ll \frac{\|f\|_2^2}{m}$.
\end{proof}
%\begin{Rem}
%If $\{h_t\}$ is SD for $\mathcal{X}$, then the same argument implies that for any $f\in L^2(\mathcal{X},\mu)$
%\begin{equation}
%\label{bmet}\|\beta_m^+(f)-\mu(f)\|_2\ll\frac{\|f\|_2}{m^{\frac12}}.
%\end{equation}
%\end{Rem}

\subsection{Groups of real rank one}
For groups of real rank one without property $(T)$ we have to take into account the contribution of the possible exceptional spectrum.
The argument is similar to the one used in \cite[Theorem 15]{Kelmer17b} for the orthogonal groups, and we include the details for the reader's convenience.
Doing this leads to the following.
\begin{Prop}
\label{ewot}Let $G$ be locally isomorphic to $SO(d+1,1)$ or $SU(d,1)$ with $d\geq 2$. Then for any unbounded discrete time one-parameter flow $H=\{h_m\}_{m\in\mathbb{Z}}$, for any sufficiently small $\epsilon> 0$ and for any spherical $f\in L^2(\mathcal{X})$ we have
$$\|\beta_m^+(f)-\mu(f)\|_2\ll_{\epsilon} \frac{\|f\|_2}{\sqrt{m}}+\sum_{s_k\in [\rho-\frac{1}{\kappa}, \rho)}\frac{|\langle f,\varphi_k\rangle|}{m^{\frac{\kappa}{2}(\rho-s_k)(1-\epsilon)}}, $$
where $\varphi_k$ are the finitely many spherical exceptional forms. When $G$ is locally isomorphic to $\SO(2,1)$ we have the same result with the first term replaced by $\frac{||f||_2}{m^{\frac12(1-\epsilon)}}$.%$\frac{\|f\|_2\log(m)}{\sqrt{m}}$.
\end{Prop}
\begin{proof}
Write $f= \langle f, 1\rangle +\sum_{k}\langle f,\varphi_k\rangle \varphi_k+ f_0$ with $f_0\in L^2_{\textrm{temp}}(\mathcal{X})$. Then
$$\|\beta_m^+(f)-\mu(f)\|_2\leq \sum_{k}|\langle f,\varphi_k\rangle \|\beta_m^+(\varphi_k)\|_2+ \|\beta_m^+(f_0)\|_2.$$
Using $(\ref{tede})$ and the same argument as in Proposition \ref{pwt}, we get
\begin{equation}
\label{tdecay}
\|\beta_m^+(f_0)\|_2\ll_\epsilon \left\{
  \begin{array}{lr}
    \frac{\|f\|_2}{m^{\frac12(1-\epsilon)}} & \textrm{$\fg=\mathfrak{so}(2,1)$}\\
    \\
    \frac{\|f\|_2}{\sqrt{m}}& \textrm{otherwise.}
  \end{array}
\right.
\end{equation}
Similarly, for each of the spherical exceptional forms $\varphi_k$ in $\pi_{s_k}$ we can bound
\[\|\beta_m^+(\varphi_k)\|_2^2\ll_{\epsilon} \frac{1}{m}(1+\sum_{k=1}^{m} \frac{1}{k^{\kappa(\rho-s_k)(1-\epsilon)}})\ll  \left\{
  \begin{array}{lr}
   \frac{1}{m} & \textrm{$s_k< \rho-\frac{1}{\kappa}$}\\
    \frac{1}{m^{\kappa(\rho-s_k)(1-\epsilon)}} &\ \textrm{$s_k\geq \rho-\frac{1}{\kappa}$.}
  \end{array}
\right.
\]
Hence
\[\|\beta_m^+(\varphi_k)\|_2\ll_{\epsilon} \left\{
  \begin{array}{lr}
   \frac{1}{\sqrt{m}} & \textrm{$s_k< \rho-\frac{1}{\kappa}$}\\
    \frac{1}{m^{\frac{\kappa}{2}(\rho-s_k)(1-\epsilon)}} &\ \textrm{$s_k\geq \rho-\frac{1}{\kappa}$.}
  \end{array}
\right.
\]
Combining this with $(\ref{tdecay})$ and using the bound $|\langle f, \varphi_k\rangle|\leq \|f\|_2$ for $s_k< \rho-\frac{1}{\kappa}$ concludes the proof.
%$$\|\beta_m^+(f)-\mu(f)\|_2\ll_{\epsilon} \frac{\|f\|_2}{m^{\frac12(1-\epsilon)}}+\sum_{s_k\in [\rho-\frac{1}{\kappa}, \rho)}\frac{|\langle f,\varphi_k\rangle|}{m^{\frac{\kappa}{2}(\rho-s_k)(1-\epsilon)}} ,$$
%where the first term could be replaced by $\frac{\|f\|_2}{m^{\frac12}}$ if $G$ is not locally isomorphic to $SO(2,1)$.
\end{proof}
\subsection{A variance estimate}
We now apply the above mean ergodic theorem for a variance estimate. Following \cite{Kelmer17b}, for any integer $m\geq 1$ and $f\in L^2(\mathcal{X})$, we define the set
$$\cC_{m,f}=\left\{x\in \mathcal{X}\ |\ \left|\beta^+_m(f)(x)-\mu(f)\right|\geq \tfrac{\mu(f)}{2}\right\}.$$
Applying the mean ergodic theorem we obtain the following estimate for the measure of $\cC_{m,f}$ when $f$ is an indicator function of a spherical set.
\begin{Prop}
\label{est}
For any spherical set $B\subset \mathcal{X}$ and $f=\chi_B$ its indicator function:
\begin{enumerate}
\item[(1)] If the flow is SD then
\begin{equation}
\label{estwt2}\mu(\mathcal{C}_{m,f})\ll\frac{1}{m\mu(B)}.
\end{equation}
\item[(2)] If $G$ has property (T), then for all sufficiently small $\epsilon>0$
\begin{equation}
\label{estwt1}\mu(\mathcal{C}_{m,f})\ll_{\epsilon}\frac{1}{m^{1-\epsilon}\mu(B)}.
\end{equation}
\item[(3)] If $G$ is locally isomorphic to $\SO(d+1,1)$ or $\SU(d,1)$ with $d\geq 2$, then
\begin{equation}
\label{estwot1}\mu(\mathcal{C}_{m,f})\ll \frac{1}{m\mu(B)}+ \frac{1}{\left(m\mu(B)^{\frac23}\right)^{2\tau(\G)/3}}.
\end{equation}
\item[(4)] If $G$ is locally isomorphic to $\SO(2,1)$, then for all sufficiently small $\epsilon>0$
\begin{equation}
\label{estwot2} \mu(\mathcal{C}_{m,f})\ll_{\epsilon}\frac{1}{m^{1-\epsilon}\mu(B)}+\frac{1}{\left(m^{1-\epsilon}\mu(B)^{1+\frac{\epsilon}{2}}\right)^{2\tau(\G)}}
\end{equation}
where for the last two cases, $\tau(\G)$ denotes the spectral gap parameter defined in \eqref{e:SpectralGap}.
\end{enumerate}
\end{Prop}

\begin{proof}
Since $\mu(\cC_{m,f})\leq 1$ we may assume that $m\mu(B)\geq 1$ since otherwise the result holds trivially.

Now, by definition if $x\in \mathcal{C}_{m,f}$, then $\left|\beta_m^+(f)-\mu(B)\right|\geq \frac{\mu(B)}2$ and hence
\begin{equation}
\label{estt}\left|\left|\beta_m^+(f)-\mu(B)\right|\right|_2^2\geq \frac{1}{4}\int_{\mathcal{C}_{m,f}}\mu(B)^2d\mu(x)=\frac{\mu(\mathcal{C}_{m,f})\mu(B)^2}4.
\end{equation}

On the other hand, for SD flows by Proposition \ref{pwt} we have
\begin{equation}
\label{e3}\|\beta_m^+(f)-\mu(B)\|_2^2\ll\frac{\|f\|_2^2}{m}= \frac{\mu(B)}{m}.
\end{equation}
Combining $(\ref{estt})$ and $(\ref{e3})$ we get $(\ref{estwt2})$.

If $G$ has property (T), then again by Proposition \ref{pwt} we have
\begin{equation}
\label{e2}\|\beta_m^+(f)-\mu(B)\|_2^2\ll_{\epsilon}\frac{\|f\|_2^2}{m^{1-\epsilon}}= \frac{\mu(B)}{m^{1-\epsilon}}.
\end{equation}
Combining $(\ref{estt})$ and $(\ref{e2})$, we get $(\ref{estwt1})$.

If $G$ is locally isomorphically to $\SO(d+1,1)$ or $\SU(d,1)$ with $d\geq 2$, then by Proposition \ref{ewot} we have
$$\|\beta_m^+(f)-\mu(B)\|_2^2\ll_{\epsilon} \frac{\|f\|^2_2}{m}+\sum_{s_k\in [\rho-\frac{1}{\kappa}, \rho)}\frac{|\langle f,\varphi_k\rangle|^2}{m^{\kappa(\rho-s_k)(1-\epsilon)}}.$$
Recall that by Proposition \ref{nE}, the exceptional forms $\varphi_k\in L^p(\mathcal{X})$ for any $1< p< \frac{2\rho}{\rho-s_k}$.
For $\epsilon>0$ sufficiently small let $p_k=\frac{2\rho}{(1+\epsilon/2)(\rho-s_k)}$, $q_k=\frac{p_k-1}{p_k}$, and use H\"older inequality to bound
$$|\langle f,\varphi_k\rangle|\leq \|\vf\|_{p_k} \|f\|_{q_k}\ll_\epsilon  \|f\|_{q_k}=\mu(B)^{\frac{1}{q_k}}.$$
With this bound we get
\begin{equation}
\label{enot1}\|\beta_m^+(f)-\mu(B)\|_2^2\ll_{\epsilon}\frac{\mu(B)}{m}+ \sum_{s_k\in [\rho-\frac{1}{\kappa}, \rho)}\frac{\mu(B)^{\frac{2}{q_k}}}{m^{\kappa(\rho-s_k)(1-\epsilon)}}.
\end{equation}
Combining $(\ref{estt})$ and $(\ref{enot1})$ we get
\begin{equation}\label{arro}
\mu(\mathcal{C}_{m,f})\ll_{\epsilon} \frac{1}{m\mu(B)}+ \sum_{s_k\in[\rho-\frac{1}{\kappa},\rho)}\frac{1}{\left(m^{\kappa(1-\epsilon)}\mu(B)^{\frac{1+\frac{1}{2}\epsilon}{\rho}}\right)^{\rho-s_k}}.  \end{equation}
%If $G$ is locally isomorphic to $SO(2,1)$, then $\kappa= 2$, $\rho=\frac12$ and this gives $(\ref{estwot2})$. If $G$ is locally isomorphic to $SO(d+1,1)$ or $SU(d,1)$ with $n\geq 2$. Then by Proposition \ref{ewot}, the first term of the right hand-side of $(\ref{arro})$ can be replaced by $\frac{1}{m\mu(B)}$.
Note that we always have that $\kappa\rho\geq 2$ so fixing $\epsilon<1/3$ sufficiently small so that $\frac{1+\frac12\epsilon}{\kappa\rho(1-\epsilon)}<\frac23$  we get
\begin{align*}
\mu(\mathcal{C}_{m,f})&\ll \frac{1}{m\mu(B)}+ \sum_{s_k\in[\rho-\frac{1}{\kappa},\rho)}\frac{1}{\left(m^{\kappa(1-\epsilon)}\mu(B)^{\frac{1+\frac{1}{2}\epsilon}{\rho}}\right)^{\rho-s_k}}\\
&=\frac{1}{m\mu(B)}+\sum_{s_k\in[\rho-\frac{1}{\kappa},\rho)}\frac{1}{\left(m\mu(B)^{\frac{1+\frac12\epsilon}{2\kappa\rho/3}}\right)^{\kappa(\rho-s_k)(1-\epsilon)}}\\
&<\frac{1}{m\mu(B)}+\sum_{s_k\in[\rho-\frac{1}{\kappa},\rho)}\frac{1}{\left(m\mu(B)^{\frac23}\right)^{2\kappa(\rho-s_k)/3}}.
\end{align*}
Finally, note that for all terms in the sum we have that $\tau(\G)\leq \rho-s_k\leq \frac{1}{\kappa}$ so that
\begin{align*}
\mu(\mathcal{C}_{m,f})&\ll\frac{1}{m\mu(B)}+\frac{1}{\left(m\mu(B)^{\frac23}\right)^{2\tau(\Gamma)/3}}.
\end{align*}

For $\SO(2,1)$ we have that $\kappa= 2$ and $\rho=\frac12$ and a similar argument gives \eqref{estwot2}
\end{proof}

\section{Applications to shrinking target problems}
We now combine all the ingredients and apply them to the shrinking target problems we described in the introduction. As before, throughout this section we let $G$ denote a connected semisimple Lie group with finite center and no compact factors, $\G\leq G$ an irreducible lattice $\{h_t\}_{t\in \R}$ an unbounded one-parameter subgroup, generating a discrete time flow given by the action of the discrete group $H=\{h_m\}_{m\in \bbZ}$ on the homogenous space $\cX=\G\bk G$.

\subsection{The hitting time problem}
Let $\{B_t\}_{t>0}$ be a monotone family of spherical shrinking targets in $\cX$.
We recall that by Lemma \ref{ght} and \ref{lower bound}, in order to prove Theorem \ref{t:hittingtime} it is enough to show that the estimate \eqref{up} holds. We show this in the following Lemma.
\begin{Lem}
Assume that either $G$ has property (T) or that it is simple of real rank one. Then $(\ref{up})$ holds for any $0<\delta<\frac12$.
\end{Lem}

\begin{proof}
Given a spherical measurable set, $B\subset \mathcal{X}$, let $f=\chi_B$ denote its indicator function, and for any $m\in \N$ let $\cM^i_{m,B}$ be as defined in \eqref{e:missing}. In particular, for any $x\in\mathcal{M}_{m,B}^i$ we have that $\beta_m^+(f)(x)=\frac{k}{m}$ for some $0\leq k< i$. Thus for any $m>\frac{2i}{\mu(B)}$, if $x\in \cM^i_{m,B}$ then
$$|\beta_m^+(f)(x)-\mu(f)|=\mu(B)-\frac{k}{m}> \mu(B)-\frac{i}{m}> \frac12\mu(B),$$
so $\mathcal{M}_{m,B}^i\subseteq \mathcal{C}_{m,f}$.

We recall that $\mathcal{N}$ is the set of integers $\ell\geq 0$ such that $\{\mu(B_t)\ |\ t> 0\}\cap [\frac{1}{2^{\ell}}, \frac{1}{2^{\ell+1}})$ is nonempty, and for each $\ell\in\mathcal{N}$, $\underline{B}_{\ell}=\bigcap_{\frac{1}{2^{\ell+1}}\leq\mu(B_t)<\frac{1}{2^{\ell}}}B_t$ with measure $\frac{1}{2^{\ell+1}}\leq \mu\left(\underline{B}_{\ell}\right)\leq \frac{1}{2^{\ell}}$. For each $\ell\in\mathcal{N}$, let $f_{\ell}=\chi_{\underline{B}_{\ell}}$. Fix $0<\delta<\frac12$ and let $j=\floor{\frac{\log 4i}{\delta\log 2}}$. Then for any $\ell\in \mathcal{N}\cap[j,\infty)$, $\mathcal{M}^i_{\floor{2^{\ell(1+\delta)}},\underline{B}_{\ell}}\subset \mathcal{C}_{\floor{2^{\ell(1+\delta)}}, f_{\ell}}$.
Hence to prove $(\ref{up})$ it is enough to show that
$$\sum_{\ell\in\mathcal{N}\cap\left[j,\infty\right)}\mu\left(\mathcal{C}_{\floor{2^{\ell(1+\delta)}}, f_{\ell}}\right)<\infty.$$

We now use Proposition \ref{est} to estimate $\mu\left(\mathcal{C}_{\floor{2^{\ell(1+\delta)}}, f_{\ell}}\right)$.
First, when $G$ has property $(T)$, Proposition \ref{est} with $\epsilon=\frac{\delta}{2(1+\delta)}$ gives that
$$\mu\left(\mathcal{C}_{\floor{2^{\ell(1+\delta)}}, f_{\ell}}\right)
\ll\frac{1}{\floor{2^{\ell(1+\delta)}}^{1-\epsilon}\mu\left(\underline{B}_{\ell}\right)}
%\asymp \frac{2^{\ell}}{2^{\ell(1+\delta)(1-\epsilon)}}\asymp 2^{(\epsilon-\delta+\epsilon\delta)\ell}.
\ll 2^{-\delta \ell/2},
$$
is summable.
%Take $\epsilon=\frac{\delta}{2(1+\delta)}$ such that $\epsilon-\delta+\epsilon\delta=-\frac12\delta$, then $\mu\left(\mathcal{C}_{\floor{2^{\ell(1+\delta)}}, f_{\ell}}\right)\ll_{\delta}2^{-\frac12\delta\ell}$ is summable.
Next, for $G$ locally isomorphic to $SO(d+1,1)$ or $SU(d,1)$ with $d\geq 2$, Proposition \ref{est} implies that
\begin{align*}
\mu\left(\mathcal{C}_{\floor{2^{\ell(1+\delta)}}, f_{\ell}}\right)&
\ll_{\epsilon}
\frac{1}{\floor{2^{\ell(1+\delta)}}\mu\left(\underline{B}_{\ell}\right)}+ \frac{1}{\left(\floor{2^{\ell(1+\delta)}}\mu\left(\underline{B}_{\ell}\right)^{\frac23}\right)^{2\tau(\G)/3}}\\
&\ll 2^{-\delta\ell}+ 2^{-2\tau(\G)\ell/9}
\end{align*}
is also summable.
Finally, when $G$ is locally isomorphic to $SO(2,1)$, Proposition \ref{est} with $\epsilon=\frac{\delta}{3+2\delta}$ implies that
%\begin{align*}
%\mu\left(\mathcal{C}_{\floor{2^{\ell(1+\delta)}}, f_{\ell}}\right)%&\ll_{\epsilon}\frac{1}{\floor{2^{\ell(1+\delta)}}^{1-\epsilon}\mu\left(\underline{B}_{\ell}\right)}+\sum_{s_k\in(0, \frac12)}\frac{1}{\left(\floor{2^{\ell(1+\delta)}}^{1-\epsilon}\mu\left(\underline{B}_{\ell}\right)^{1+\frac{\epsilon}{2}}\right)^{1-2s_k}}\\
%&\asymp 2^{(\epsilon-\delta+\epsilon\delta)\ell}+ \sum_{s_k\in (0,\frac12)}2^{(\frac32\epsilon-\delta+\epsilon\delta)(1-2s_k)}
%\end{align*}
%Take $\epsilon=\frac{\delta}{3+2\delta}$. Then $\frac32\epsilon-\delta+\epsilon\delta=-\frac12\delta$ and $\epsilon-\delta+\epsilon\delta=-\frac12\delta-\frac12\epsilon< -\frac12\delta$.
$$\mu\left(\mathcal{C}_{\floor{2^{\ell(1+\delta)}}, f_{\ell}}\right)\ll_{\delta}2^{-\delta\ell/2}+2^{-\delta\tau(\Gamma)\ell}$$
is summable.
\end{proof}

\subsection{Logarithm laws}
We now apply our results to the special case where  $\G\bk G$ is not compact and the shrinking targets are cusp neighborhoods, to prove Corollary \ref{loglaws}, establishing logarithm laws for one-parameter flows.
Since $\G\bk G$ is not compact when $G$ is of higher rank, we may assume that $\G$ is a congruence group. We also assume here that $G$ is not locally isomorphic to a product of copies of $\SL_2(\R)$ and $\SL_2(\bbC)$, and note that, in that case, Corollary \ref{loglaws} already follows from  \cite{KelmerMohammadi12}.
%We do not assume here that $G$ is of rank one or has $(T)$ , however, since $\G\bk G$ is not compact when $G$ is of higher rank, we may assume that $\G$ is a congruence group. We also assume here that $G$ is not locally isomorphic to a product of copies of $\SL_2(\R)$ and $\SL_2(\bbC)$, and note that, in that case, Corollary \ref{loglaws} already follows from  \cite{KelmerMohammadi12}. Finally note that under these assumptions if $G$ is not of rank one, then $\G$ is a congruence group and $G$ has no factor locally isomorphic to $\SL_2(\R)$ so, following the argument in Remarks \ref{r:conditional} and \ref{r:conditional2}, the conclusion of Theorem \ref{t:hittingtime} holds in this more general setting.

Now, fix a $K\times K$-invariant function $d(\cdot,\cdot)$  on $\Gamma\backslash G\times \G\bk G$ (coming from a distance function on $\G\bk \cH$) satisfying $\eqref{e:cuspdecay}$ as in the introduction. For a fixed reference point $x_0\in \cX$ let
$$B_t=\{x\in \mathcal{X}\ |\ d(x,x_0)> t\}$$
be the corresponding spherical cusp neighborhood. We first note that the easy half of Borel-Cantelli lemma, together with a standard continuity argument, implies that the upper bound,
$$\limsup_{t\to\infty}\frac{d(xh_t,x_0)}{\log t}\leq \frac{1}{\varkappa},$$
holds for a.e. $x\in \mathcal{X}$. For the lower bound,  note that by Theorem \ref{t:hittingtime}, when $G$ has property $(T)$ or is of real rank one the limit $(\ref{hit})$ holds for a.e. $x\in \cX$. When $G$ is of higher rank without property $(T)$, since $\Gamma$ is a congruence subgroup and $G$ is not locally isomorphic to a product of copies of $SL_2(\R)$ and $SL_2(\bbC)$, there is no factor of $G$ locally isomorphic to $SL_2(\R)$. Thus, using the known bounds towards the Selberg-Ramanujan conjecture (see Remarks \ref{r:conditional} and \ref{r:conditional2}) the limit \eqref{t:hittingtime} still holds for a.e. $x\in \cX$.
The proof of Corollary \ref{loglaws} now follows immediately from the following Lemma.
\begin{Lem}
For any $x\in \cX$ satisfying $(\ref{hit})$ we have that
$$\limsup_{t\to\infty}\frac{d(xh_t,x_0)}{\log t}\geq \frac{1}{\varkappa}.$$
\end{Lem}
\begin{proof}
By $(\ref{hit})$ and $(\ref{e:cuspdecay})$, for any small $\epsilon>0$, there exists $t_0>0$ such that for any $t\geq t_0$ we have that $\frac{\log\tau_{B_t}(x)}{-\log(\mu(B_t))}< 1+\epsilon$ and $\frac{-\log\mu(B_t)}{t}< \varkappa+\epsilon$, hence, $\frac{\log\tau_{B_t}(x)}{t}<(\varkappa+\epsilon)(1+\epsilon)$. Moreover, by the minimality of $\tau_{B_t}(x)$ we have that $xh_{\tau_{B_t}(x)}\in B_t$, or equivalently that $d(xh_{\tau_{B_t(x)}},x_0)> t$. For any integer $\ell\in \N$ let $s_{\ell}=\tau_{B_{\ell}}(x)$. The condition that $x$ satisfies \eqref{hit}, implies that $s_{\ell}\to \infty$ as $\ell\to\infty$, and for any $\ell\geq t_0$, we have
$$\frac{d(xh_{s_{\ell}}, x_0)}{\log s_{\ell}}=\frac{d(xh_{\tau_{B_{\ell}}(x)}, x_0)}{\log \tau_{B_{\ell}}(x)}> \frac{d(xh_{\tau_{B_{\ell}}(x)}, x_0)}{(\varkappa+\epsilon)(1+\epsilon)\ell}> \frac{1}{(\varkappa+\epsilon)(1+\epsilon)}.$$
To conclude, for any $x$ satisfying $(\ref{hit})$ and for any $\epsilon> 0$ we can find an unbounded sequence $\{s_{\ell}\}$ such that $\frac{d(xh_{s_{\ell}}, x_0)}{\log s_{\ell}}>\frac{1}{(\varkappa+\epsilon)(1+\epsilon)}$, implying that $\limsup_{t\to\infty}\frac{d(xh_t,x_0)}{\log t}\geq \frac{1}{\varkappa}$.
\end{proof}
\begin{Rem}\label{r:sphericalcusps}
In \cite{KleinbockMargulis1999,AthreyaMargulis14}, the cusp neighborhoods are defined by a distance function $\tilde{d}(\cdot,\cdot)$ on $\G\bk G$ induced from a left $G$-invariant and bi-$K$-invariant metric (resp. norm like pseudo metric) on $G$, rather than a distance function on $\G\bk \cH$ as in our case. In order to apply the above argument in this case we can replace the non spherical distance function $\tilde{d}(x, x_0)$ by the spherical distance function $d(x, x_0):=\inf_{k\in K}\tilde d(xk,x_0)$. Since $\tilde d(\cdot, \cdot)$ is induced from a bi-$K$-invariant metric, there exists some constant $C>0$ such that for any $x\in \mathcal{X}$, $0\leq \tilde{d}(x,x_0)-d(x,x_0)\leq C$. Thus the corresponding cusp neighborhoods satisfy  $\tilde B_{t+C}\subset B_t\subset \tilde B_t$ for any $t>0$, where $\tilde{B_t}=\{x\in\mathcal{X}\ |\ \tilde{d}(x,x_0)> t\}$. Now from \cite{KleinbockMargulis1999,AthreyaMargulis14} we have that  $\mu(\tilde B_t)\asymp e^{-\varkappa t}$ for some $\varkappa>0$ and hence also have that $\mu(B_t)\asymp e^{-\varkappa t}$. Since we also have $\tilde{d}(x,x_0)=d(x,x_0)+O(1)$, we get that for a.e. $x\in \mathcal{X}$
$$\limsup_{t\to\infty}\frac{\tilde{d}(xh_t,x_0)}{\log t}=\limsup_{t\to\infty}\frac{d(xh_t,x_0)}{\log t}= \frac{1}{\varkappa}.$$
\end{Rem}

\subsection{Orbits eventually always hitting}
We now give the proof of Theorem \ref{t:ae}.
Given a monotone sequence of spherical shrinking targets $\{B_m\}_{m\in\N}$, let $f_m=\chi_{B_m}$.
%Recalling Lemma \ref{l:EventuallyHitting}  and noting that $\cM^1_{B,m)\subseteq \cC_{f,m}$  $(\ref{e:summable})$ implies $\sum_{j=1}^{\infty}\mu(\mathcal{C}_{2^{j\pm 1}, f_{2^j}})<\infty$.
By Proposition \ref{est}, for SD flows
$$\mu\left(\mathcal{C}_{2^{j\pm 1},f_{2^j}}\right)\ll\frac{1}{2^{j\pm 1}\mu(B_{2^j})}\asymp \frac{1}{2^j\mu(B_{2^j})}.$$
and if $G$ is locally isomorphic to $\SO(d+1,1)$ or $\SU(d,1)$ with $d\geq 2$, then
\begin{align*}
\mu\left(\mathcal{C}_{2^{j\pm 1},f_{2^j}}\right)&\ll \frac{1}{2^j\mu(B_{2^j})}+\frac{1}{(2^{j}\mu(B_{2^j})^{\frac23})^{2\tau(\G)/3}}\\
&< \frac{1}{2^j\mu(B_{2^j})}+\frac{2^{-2j/9}}{(2^{j}\mu(B_{2^j}))^{4\tau(\G)/9}}.
\end{align*}
Since $\sum_{j=0}^{\infty}\frac{1}{2^j\mu(B_{2^j})}<\infty$, there exists $j_0$ such that for any $j\geq j_0$, $2^j\mu(B_{2^j})>1$. Hence for all $j\geq j_0$, we have
$$\mu\left(\mathcal{C}_{2^{j\pm 1},f_{2^j}}\right)\ll_{\epsilon}\frac{1}{2^j\mu(B_{2^j})}+2^{-2j/9}.$$
Thus in both cases, \eqref{e:summable} implies that
$\sum_{j=0}^{\infty}\mu\left(\mathcal{C}_{2^{j\pm 1},f_{2^j}}\right)<\infty$, and since  $\cM^1_{m,B}\subseteq \cC_{m,f}$ with $f= \chi_B$, then
$\sum_{j=0}^{\infty}\mu\left(\cM^1_{2^{j\pm 1},B_{2^j}}\right)<\infty$ and Theorem \ref{t:ae} follows from Lemma \ref{l:EventuallyHitting}.

\subsection{Dynamical Borel-Cantelli for SD flows.}
We now give the proof of Theorem \ref{t:SBC}, by showing that rate of decay of matrix coefficients for SD flows is sufficient to show that for any sequence of spherical shrinking targets,  $\{B_m\}_{m\in\mathbb{N}}$,
the family of functions $f_m(x)=\chi_{B_m}(xh_m)$ satisfy condition \eqref{keybound} of Lemma \ref{gbc}.
\begin{Prop}
\label{quasi.prop} Let $\{B_m\}_{m\in\mathbb{N}}$ denote any sequence of spherical sets in $\mathcal{X}$ and let $f_m(x)=\chi_{B_m}(xh_m)$.
If the flow is SD then
%for $\mathcal{X}$ or $G$ is locally isomorphic to either $SO(d+1,1)$ or $SU(d,1)$ with $n\geq 2$, $\left\{\mu(B_m)\right\}_{m\in\N}$ is nonincreasing and $\left\{m\mu(B_m)\right\}_{m\in\mathbb{N}}$ is bounded, then
there exists some constant $C>0$ such that for all $n> m\geq 1$
$$\int_{\mathcal{X}}\left(\sum_{i=m}^nf_i(x)-\sum_{i=m}^n\mu(f_i)\right)^2d\mu(x)\leq C\sum_{i=m}^n\mu(f_i).$$
\end{Prop}
\begin{proof}
We first note that by direct computation
$$\int_{\mathcal{X}}\left(\sum_{i=m}^nf_i(x)-\sum_{i=m}^n\mu(f_i)\right)^2d\mu(x)=\sum_{m\leq i,j\leq n}\left(\int_{\mathcal{X}}f_i(x)f_j(x)d\mu(x)-\mu(f_i)\mu(f_j)\right).$$
To simplify notation, let $\mu_i=\mu(f_i)=\mu(B_i)$ and $\mu_{i,j}=\int_{\mathcal{X}}f_i(x)f_j(x)d\mu(x)$, and note that $\mu_{i,i}=\mu_i$ (since for $f_i(x)=\chi_{B_i}(xh_i)$, $f_i^2=f_i$). With these notations we have
\begin{align*}
\int_{\mathcal{X}}\left(\sum_{i=m}^nf_i(x)-\sum_{i=m}^n\mu(f_i)\right)^2d\mu(x)&=\sum_{i=m}^n\mu_i+\sum_{m\leq i\neq j\leq n}\mu_{i,j}-\sum_{1\leq i,j\leq n}\mu_i\mu_j\\
  &\leq\sum_{i=m}^n\mu_i+\sum_{m\leq i\neq j\leq n}(\mu_{i,j}-\mu_i\mu_j).
\end{align*}
Thus it suffices to show that
$$\sum_{m\leq i\neq j\leq n}(\mu_{i,j}-\mu_i\mu_j)\ll \sum_{i=m}^n\mu_i.$$
Since we assume the flow is SD, there exists some constant $\eta>1$ such that for any spherical $\varphi,\psi \in  L^2_0(\G\bk G)$, for all $|t|\geq 1$
$$|\langle \pi(h_t)\varphi, \psi\rangle |\ll \frac{\|\varphi\|_2\|\psi\|_2}{|t|^{\eta}}.$$
In particular, for any $i\neq j$, taking $\varphi=\chi_{B_i}-\mu(B_i)$ and $\psi=\chi_{B_j}-\mu(B_j)$ we get that
%$$\left|\langle \pi(h_i)\varphi,\pi(h_j)\psi\rangle\right| =\left|\langle f_i-\mu(f_i), f_j-\mu(f_j)\rangle\right|= \left|\mu_{ij}-\mu_i\mu_j\right|.$$
%On the other hand,
\begin{align*}
 \left|\mu_{i,j}-\mu_i\mu_j\right|&=\left|\langle \pi(h_i)\varphi, \pi(h_j)\psi\rangle\right|\\
		 &=\left|\langle\pi(h_{i-j})\varphi, \psi\rangle\right|\\
                 &\ll \frac{\|\varphi\|_2\|\psi\|_2}{|i-j|^{\eta}}
                 \leq \frac{\sqrt{\mu_i\mu_j}}{|i-j|^{\eta}}.
\end{align*}
%Thus for any $i\neq j$, $|\mu_{ij}-\mu_i\mu_j|\ll \frac{\sqrt{\mu_i\mu_j}}{|i-j|^{\eta}}$.
It now suffices to show that
$$\sum_{m\leq i\neq j\leq n}\frac{\sqrt{\mu_i\mu_j}}{|j-i|^{\eta}}\ll \sum_{i=m}^n \mu_i.$$
We rewrite the sum on the left as
$$\sum_{m\leq i\neq j\leq n}\frac{\sqrt{\mu_i\mu_j}}{|j-i|^{\eta}}
=\sum_{\substack{m\leq i\neq j\leq n \\ \mu_i\leq \mu_j}}\frac{\sqrt{\mu_i\mu_j}}{|j-i|^{\eta}}+ \sum_{\substack{m\leq i\neq j\leq n \\ \mu_i> \mu_j}}\frac{\sqrt{\mu_i\mu_j}}{|j-i|^{\eta}},$$
and using symmetry we can bound
\begin{align*}
\sum_{m\leq i\neq j\leq n}\frac{\sqrt{\mu_i\mu_j}}{|j-i|^{\eta}}&\leq 2\sum_{\substack{m\leq i\neq j\leq n \\ \mu_i\leq \mu_j}}\frac{\sqrt{\mu_i\mu_j}}{|j-i|^{\eta}}\\
     &\leq 2\sum_{\substack{m\leq i\neq j\leq n \\ \mu_i\leq \mu_j}}\frac{\mu_j}{|j-i|^{\eta}}\\
     &= 2\sum_{m\leq j\leq n}\mu_j(\sum_{\substack{m\leq i\leq n \\ \mu_i\leq \mu_j, i\neq j}}\frac{1}{|j-i|^{\eta}}).
\end{align*}
We can bound the inner sum by the convergent series $2\sum_{i=1}^{\infty}\frac{1}{i^{\eta}}$
%$$\sum_{\substack{m\leq i\leq n \\ i\neq j \\ \mu_i\leq \mu_j}}\frac{1}{|j-i|^{\eta}}
%\leq \sum_{\substack{m\leq i\leq n \\ i\neq j}}\frac{1}{|j-i|^{\eta}}
%\leq  2\sum_{i=1}^{\infty}\frac{1}{i^{\eta}},$$
%and since $\eta>1$ the series on the right converges,
thus concluding the proof.
%which is uniformly bounded since $\eta>1$. Hence $\sum_{m\leq i\neq j\leq n}\frac{\sqrt{\mu_i\mu_{ij}}}{|j-i|^{\eta}}\ll \sum_{i=m}^n \mu_j$ as claimed.
\end{proof}

\begin{proof}[Proof of Theorem \ref{t:SBC}]
Let $\{B_m\}_{m\in\mathbb{N}}$ denote any sequence of spherical sets in $\mathcal{X}$ satisfying that $\sum_{m=1}^{\infty}\mu(B_m)=\infty$.
For $f_m(x)=\chi_{B_m}(xh_m)$ we have that
$$\sum_{1\leq j\leq m}f_j(x)=\#\{1\leq j\leq m: xh_j\in B_j\},$$ and hence, Proposition \ref{quasi.prop} and Lemma \ref{gbc} imply that for a.e. $x\in \mathcal{X}$
\begin{displaymath}
\lim_{m\to\infty}\frac{\#\{1\leq j\leq m: xh_j\in B_j\}}{\sum_{1\leq j\leq m} \mu(B_j)}= 1.\qedhere
\end{displaymath}
\end{proof}

\subsection{Dynamical Borel-Cantelli for rank one groups}
We now turn to the case where $G$ is locally isomorphic to $\SO(d+1,1)$ or $\SU(d,1)$ with $d\geq 2$
(if $G$ is of real rank one with property $(T)$, then every unbounded flow is SD and we can apply Theorem \ref{t:SBC}). Combining bounds coming from the mean ergodic theorem, together with Lemma \ref{gbc} we show that any monotone sequence of spherical shrinking targets, $\{B_m\}_{m\in \N}$, is BC. In this case, the argument is different when the sequence $\{m\mu(B_m)\}_{m\in\N}$ is bounded or unbounded, and we treat these cases separately.
We first deal with the bounded case by showing the following.

\begin{Prop}\label{p:bounded}
For $G$ locally isomorphic to $\SO(d+1,1)$ or $\SU(d,1)$ with $d\geq 2$, let $\{B_m\}_{m\in \N}$ denote a sequence of spherical sets satisfying that $\left\{m\mu(B_m)\right\}_{m\in\N}$ is uniformly bounded. Let $f_m(x)=\chi_{B_m}(xh_m)$,
then there exists some constant $C>0$ such that for all $n> m\geq 1$
$$\int_{\mathcal{X}}\left(\sum_{i=m}^nf_i(x)-\sum_{i=m}^n\mu(f_i)\right)^2d\mu(x)\leq C\sum_{i=m}^n\mu(f_i).$$
\end{Prop}
\begin{proof}
As in the proof of Proposition \ref{quasi.prop}, denote by $\mu_i=\mu(f_i)=\mu(B_i)$ and $\mu_{i,j}=\int_{\mathcal{X}}f_i(x)f_j(x)d\mu(x)$.
Using the spectral decomposition we can write
\begin{displaymath}
\chi_{B_m}=\mu(B_m)+\sum_{k}\langle \chi_{B_m},\varphi_k\rangle \varphi_k + f_m^0,
\end{displaymath}
with $f_m^{0}\in L^2_{\textrm{temp}}(\G\bk \cH)$. Hence for any $i\neq j$,
\begin{eqnarray*}
\mu_{i,j}&=&\langle f_i, f_j\rangle=\langle \pi(h_i)\chi_{B_i}, \pi(h_j)\chi_{B_j}\rangle=\langle \pi(h_{i-j})\chi_{B_i}, \chi_{B_j}\rangle\\
         &=&\mu_i\mu_j+ \sum_k\langle\chi_{B_i},\varphi_k\rangle\overline{\langle \chi_{B_j},\varphi_k\rangle}\langle\pi(h_{i-j})\varphi_k,\varphi_k\rangle+ \langle\pi(h_{i-j})f_i^0, f_j^0\rangle.
\end{eqnarray*}
Thus by Proposition \ref{rod} we have for any small $\epsilon>0$
%\begin{eqnarray*}
$$
|\mu_{i,j}-\mu_i\mu_j|\ll_{\epsilon} %&\sum_k\frac{|\langle\chi_{B_i},\varphi_k\rangle\langle \chi_{B_j},\varphi_k\rangle|}{|i-j|^{\kappa(\rho-s_k)(1-\epsilon)}}+ \frac{\|f_i^0\|_2\|f_j^0\|_2}{|i-j|^{\kappa\rho(1-\epsilon)}}\\
%&\leq&
\sum_k\frac{|\langle\chi_{B_i},\varphi_k\rangle\langle \chi_{B_j},\varphi_k\rangle|}{|i-j|^{\kappa(\rho-s_k)(1-\epsilon)}}+ \frac{\sqrt{\mu_i\mu_j}}{|i-j|^{\kappa\rho(1-\epsilon)}}.
$$
%\end{eqnarray*}
%where for the second inequality we used the bound $\|f_i^0\|_2\|f_j^0\|_2\leq \|\chi_{B_i}\|_2\|\chi_{B_j}\|_2=\sqrt{\mu_i\mu_j}$.

Let $\eta=\kappa\rho(1-\epsilon)$ and note that for all $0<\epsilon<1/4$ we have that $\eta>3/2$ (since $G$ is not locally isomorphic to $SO(2,1)$). Hence, by the same arguments
as in the proof of Proposition \ref{quasi.prop}, for any $0<\epsilon<\frac14$ we can bound
$$\sum_{m\leq i\neq j\leq n}\frac{\sqrt{\mu_i\mu_j}}{|j-i|^{\kappa\rho(1-\epsilon)}}\ll \sum_{i=m}^n \mu_i.$$
%Similarly, for each of the exceptional forms, let $\eta_k=\kappa(\rho-s_k)(1-\epsilon)$ and note that for $\eta_k>1$ using the same argument with the trivial bound
%$|\langle \chi_{B_j},\varphi_k\rangle|\leq \sqrt{\mu_j}$ we get that
%$$\sum_{m\leq i\neq j\leq n}\frac{|\langle\chi_{B_i},\varphi_k\rangle\langle \chi_{B_j},\varphi_k\rangle|}{|i-j|^{\eta_k}}\ll \sum_{i=m}^n \mu_i.$$
Hence, it suffices to show that for each of the finitely many exceptional forms we have
$$\sum_{m\leq i\neq j\leq n}\frac{|\langle\chi_{B_i},\varphi_k\rangle\langle \chi_{B_j},\varphi_k\rangle|}{|i-j|^{\kappa\rho(1-\epsilon)}}\ll_{\epsilon} \sum_{i=m}^n\mu_i.$$
Now let $0<\epsilon_0<\frac14$ be sufficiently small such that $\frac{2\rho}{(\rho+s_k)(1-\epsilon_0)}<2$ for all $s_k$. For $0<\epsilon\leq \epsilon_0$ (to be determined later) let $q_k=\frac{2\rho}{(\rho+s_k)(1-\epsilon)}$ and let $\eta_k=\kappa(\rho-s_k)(1-\epsilon)$. Note that $\frac12<\frac{1}{q_k}<1$.
Recalling that the exceptional form $\varphi_k\in\pi_{s_k}$ is contained in $L^p(\mathcal{X})$ for any $1<p<\frac{2\rho}{\rho-s_k}$ we can bound
%Thus we can bound $|\langle \chi_{B_i}, \varphi_k\rangle \langle \chi_{B_j},\varphi_k\rangle|\ll_{q_k} \mu_i^{\frac{1}{q_k}}\mu_j^{\frac{1}{q_k}}$ for any $q_k=\frac{p}{p-1}> \frac{2\rho}{\rho+s_k}$.
%$q_k=\frac{2\rho}{(\rho+s_k)(1-\epsilon)}$ with $0<\epsilon<\frac12$ is sufficiently small so that  $\frac{2\rho}{(\rho+s_k)(1-\epsilon_0)}<2$ for all exceptional forms, and bound
$$|\langle \chi_{B_i}, \varphi_k\rangle \langle \chi_{B_j},\varphi_k\rangle|\ll_{\epsilon} \mu_i^{\frac{1}{q_k}}\mu_j^{\frac{1}{q_k}},$$
and hence together with symmetry we can bound
\begin{displaymath}
\sum_{m\leq i\neq j\leq n}\frac{|\langle\chi_{B_i},\varphi_k\rangle\langle \chi_{B_j},\varphi_k\rangle|}{|i-j|^{\kappa(\rho-s_k)(1-\epsilon)}}%&=&2\sum_{M\leq i< j\leq N}\frac{|\langle\chi_{B_i},\varphi_k\rangle\langle \chi_{B_j},\varphi_k\rangle|}{(j-i)^{\kappa(\rho-s_k)(1-\epsilon)}}\\
\ll \sum_{m\leq i\neq j\leq n}\frac{\mu_i^{\frac{1}{q_k}}\mu_j^{\frac{1}{q_k}}}{|j-i|^{\eta_k}}\\
\ll \mathop{\sum_{m\leq i\neq j\leq n}}_{\mu_j\leq \mu_i}\frac{\mu_i^{\frac{1}{q_k}}\mu_j^{\frac{1}{q_k}}}{|j-i|^{\eta_k}}.\\
\end{displaymath}
Since $\frac{1}{q_k}<1$, for $\mu_j\leq \mu_i$ we can bound
$\mu_i^{\frac{1}{q_k}}\mu_j^{\frac{1}{q_k}}\leq \mu_i\mu_j^{\frac{2}{q_k}-1}$
to get that
\begin{align*}
\mathop{\sum_{m\leq i\neq j\leq n}}_{\mu_j\leq \mu_i}\frac{\mu_i^{\frac{1}{q_k}}\mu_j^{\frac{1}{q_k}}}{|j-i|^{\eta_k}}\leq \mathop{\sum_{m\leq i\neq j\leq n}}_{\mu_j\leq \mu_i}\frac{\mu_i \mu_j^{\frac{2}{q_k}-1}}{|j-i|^{\eta_k}}.
\end{align*}
Now using the assumption that  $\{\ell \mu_{\ell}\}_{\ell\in\N}$ is uniformly bounded, (and noting that $\tfrac{2}{q_k}>1$) we can bound
$\mu_j^{\frac{2}{q_k}-1}\ll
\frac{1}{j^{\frac{2}{q_k}-1}}$ so that
$$\mathop{\sum_{m\leq i\neq j\leq n}}_{\mu_j\leq \mu_i}\frac{\mu_i \mu_j^{\frac{2}{q_k}-1}}{|j-i|^{\eta_k}}\ll \sum_{m\leq i\leq n}\mu_i\left(\mathop{\sum_{m\leq j\leq n}}_{\mu_j\leq \mu_i, j\neq i}\frac{1}{j^{\frac{2}{q_k}-1}|j-i|^{\eta_k}}\right).$$

%\begin{align*}
% \sum_{m\leq i< j\leq n}\frac{\mu_i^{\frac{1}{q_k}}\mu_j^{\frac{1}{q_k}}}{(j-i)^{\eta_k}} &\leq \sum_{m\leq i<j\leq n}\frac{\mu_i\mu_{j-i}^{\frac{2}{q_k}-1}}{(j-i)^{\eta_k}}\\
%&=\sum_{m\leq i<n}\mu_i\left(\sum_{i<j\leq n}\frac{\mu_{j-i}^{\frac{2}{q_k}-1}}{(j-i)^{\eta_k}}\right)\\
%&=\sum_{m\leq i< n}\mu_i\left(\sum_{\ell=1}^{n-i}\frac{\mu_{\ell}^{\frac{2}{q_k}-1}}{\ell^{\eta_k}}\right).
%\end{align*}

%Since the inner sum is bounded by the series $\sum_{\ell=1}^{\infty}\frac{\mu_{\ell}^{\frac{2}{q_k}-1}}{\ell^{\eta_k}}$ it is enough to show that this series converges.
%Using the assumption that  $\{\ell \mu_{\ell}\}_{\ell\in\N}$ is uniformly bounded, we can bound
%$$\sum_{\ell=1}^{\infty}\frac{\mu_{\ell}^{\frac{2}{q_k}-1}}{\ell^{\eta_k}}\ll\sum_{\ell=1}^{\infty}\frac{1}{\ell^{\frac{2}{q_k}-1+\eta_k}}.$$
For the cases we consider we have that  $\kappa-\frac{1}{\rho}\geq 1/2$ and we can estimate the exponent
$$\frac{2}{q_k}-1+\eta_k=\tfrac{(\rho+s_k)(1-\epsilon)}{\rho}-1+\kappa(\rho-s_k)(1-\epsilon)\geq(1-\epsilon)\left(2+\tfrac{\tau}{2}\right)-1,$$
where $\tau=\tau(\Gamma)$ is the spectral gap parameter for $\Gamma$.
Taking $\epsilon=\min\{\frac12\epsilon_0, \frac{\tau}{2\tau+8}\}$ we get that
$\frac{2}{q_k}-1+\eta_k\geq 1+\frac14\tau$. Hence for any $m\leq i\leq n$, the sum
\begin{align*}
\mathop{\sum_{m\leq j\leq n}}_{\mu_j\leq \mu_i, j\neq i}\frac{1}{j^{\frac{2}{q_k}-1}|j-i|^{\eta_k}}&\leq\mathop{\sum_{j> 0}}_{ j\neq i}\left(\frac{1}{j^{\frac{2}{q_k}-1+\eta_k}}+\frac{1}{|i-j|^{\frac{2}{q_k}-1+\eta_k}}\right)\\
&\ll\sum_{\ell\neq 0}\frac{1}{|\ell|^{\frac{2}{q_k}-1+\eta_k}}\leq \sum_{\ell\neq 0}\frac{1}{|\ell|^{1+\frac14\tau}}<\infty,
\end{align*}
is uniformly bounded and thus concluding the proof.
\end{proof}

Next, we consider the case where $\{m\mu(B_m)\}_{m\in\mathbb{N}}$ is unbounded. For this case we use results of the effective mean ergodic theorem to show the following.
\begin{Prop}\label{p:unbounded}
For $G$ locally isomorphic to $\SO(d+1,1)$ or $\SU(d,1)$  with $d\geq 2$, let $\{B_m\}_{m\in\mathbb{N}}$ be a monotone family of spherical shrinking targets in $\mathcal{X}=\Gamma\backslash G$ satisfying that  $\sum_{m=1}^{\infty}\mu(B_m)=\infty$ and that $\left\{m\mu(B_m)\right\}_{m\in\mathbb{N}}$ is unbounded. Then there is a subsequence $m_j$ with $m_j\mu(B_{m_j})\to \infty$ satisfying that for a.e. $x\in\mathcal{X}$
$$\lim_{j\to\infty}\frac{\#(xH^+_{m_j}\cap B_{m_j})}{m_j\mu(B_{m_j})}=1.$$
%In particular, $\#\left\{1\leq \ell\leq m_j\ |\ xg_{\ell}\in B_{m_j}\right\}\to \infty$ as $j\to\infty$.
\end{Prop}
\begin{proof}
For any $m\in \mathbb{N}$ let $f_m=\frac{\chi_{B_m}}{\mu(B_m)}$ and note that $\beta_m^+(f_m)(x)=\frac{\#(xH^+_m\cap B_m)}{m\mu(B_{m})}$. Hence it suffices to show that there is some subsequence $\left\{m_j\right\}$ such that $\beta_{m_j}(f_{m_j})(x)\to 1$ as $j\to\infty$ for a.e. $x\in \mathcal{X}$. By Proposition \ref{ewot} we have
$$\|\beta_m^+(f_m)-1\|_2\ll_{\epsilon} \frac{1}{\sqrt{m\mu(B_m)}}+\sum_{s_k\in [\rho-\frac{1}{\kappa},\rho)}\frac{\left|\langle f_m,\varphi_k\rangle\right|}{m^{\frac{\kappa}{2}(\rho-s_k)(1-\epsilon)}}.$$
Recall that we can bound $\left|\langle f_m,\varphi_k\rangle\right|\ll_q \|f_m\|_q$ for any $q> \frac{2\rho}{\rho+s_k}$. Now for $0<\epsilon< 1-\frac{1}{\kappa\rho}$, for any $s_k\in [\rho-\frac{1}{\kappa}, \rho)$ we have $q_k=\frac{1}{1-\frac{\kappa}{2}(\rho-s_k)(1-\epsilon)}> \frac{2\rho}{\rho+s_k}$ and hence $\left|\langle f_m,\varphi_k\rangle\right|\ll_{\epsilon}\|f_m\|_{q_k}=\frac{1}{\mu(B_m)^{1-\frac{1}{q_k}}}=\frac{1}{\mu(B_m)^{\frac{\kappa}{2}(\rho-s_k)(1-\epsilon)}}$. Thus for any $0<\epsilon< 1-\frac{1}{\kappa\rho}$ we have
$$\|\beta_m^+(f_m)-1\|_2\ll_{\epsilon} \frac{1}{\sqrt{m\mu(B_m)}}+\sum_{s_k\in [\rho-\frac{1}{\kappa},\rho)}\frac{1}{\left(m\mu(B_m)\right)^{\frac{\kappa}{2}(\rho-s_k)(1-\epsilon)}}.$$
Since $\left\{m\mu(B_m)\right\}_{m\in \mathbb{N}}$ is unbounded, there is a subsequence satisfying that $m_j\mu(B_{m_j})\to\infty$, for which $\|\beta_{m_j}(f_{m_j})-1\|_2\to 0$ as $j\to \infty$. Passing to another subsequence, if necessary, we get that $\beta_{m_j}(f_{m_j})(x)\to 1$ for a.e. $x\in \mathcal{X}$.
\end{proof}

We now combine the two cases to complete the proof.
\begin{proof}[Proof of Theorem \ref{t:MBC}]
Let $G$ be locally isomorphic to $\SO(d+1,1)$ or $\SU(d,1)$  with $d\geq 2$, and let $\{B_m\}_{m\in \N}$ denote a monotone family of spherical shrinking targets with $\sum_m \mu(B_m)=\infty$.
%Since we assume the sequence is monotone, in particular we have that $\{\mu(B_m)\}_{m\in\mathbb{N}}$ is nonincreasing.
Now, if the sequence $\{m\mu(B_m)\}_{m\in \N}$ is bounded then by Proposition \ref{p:bounded} we have that
for a.e. $x\in \cX$,
$$\lim_{m\to\infty} \frac{\#\{1\leq j\leq m: xh_j\in B_j\}}{\sum_{1\leq j\leq m} \mu(B_j)}=1.$$

If the sequence  $\{m\mu(B_m)\}_{m\in \N}$  is unbounded then by Proposition \ref{p:unbounded}
there is a subsequence $m_j$ such that for a.e. $x\in \cX$
$$\lim_{j\to\infty}\frac{\#\{1\leq i\leq m_j: xh_i\in B_{m_j}\}}{m_j \mu(B_{m_j})}=1.$$
In both cases, for a.e. $x\in \cX$ the set $\{m\in\N: xh_m\in B_m\}$ is unbounded, so $\{B_m\}_{m\in \N}$ is BC for this flow.
\end{proof}

\bibliographystyle{alpha}
\bibliography{DKbibliog}

\end{document}